\documentclass[12pt]{amsart}
\usepackage{amsfonts,latexsym,amsmath, amssymb}
\usepackage{mathrsfs,MnSymbol}
\usepackage{url,color}
\usepackage{upgreek}
\usepackage{fancyhdr}
\usepackage{hyperref}
\usepackage{times}
\usepackage{scalefnt}
\usepackage{url}
\usepackage{cite}
\usepackage{amsmath}

\newcommand\myuline{\bgroup\markoverwith
	{\textcolor{gray}{\rule[0.4ex]{2pt}{3pt}}}\ULon}

\newcommand{\bea}{\begin{eqnarray}}
\newcommand{\eea}{\end{eqnarray}}
\def\beaa{\begin{eqnarray*}}
\def\eeaa{\end{eqnarray*}}
\def\ba{\begin{array}}
\def\ea{\end{array}}
\def\be#1{\begin{equation} \label{#1}}
\def\eeq{\end{equation}}

\def\lV{\lVert}
\def\rV{\rVert}

\def\div{\mathrm{div}}
\def\d{\partial}

\newcommand{\Lf}{L^{\infty}}

\def\be{{\beta}}

\def\ep{\epsilon}
\def\ka{\kappa}

\def\Om{\Omega}

\def\ka{\kappa}
\def\nab{\nabla}

\def\FF{{\mathcal F}}

\def\R{{\mathbb{R}}}
\def\C{{\mathbb{C}}}

\def\N{{\mathbb N}}
\def\Z{{\mathbb{Z}}}

\def\P{{\mathbb P}}
\def\U{{\mathbb U}}

\newcommand{\<}{  \langle   }
\renewcommand{\>}{  \rangle   }

\newtheorem{theorem}{Theorem}[section]
\newtheorem{lemma}[theorem]{Lemma}
\newtheorem{proposition}[theorem]{Proposition}
\newtheorem{corollary}[theorem]{Corollary}

\newtheorem{remark}[theorem]{Remark}

\numberwithin{equation}{section}

\numberwithin{equation}{section}

\topmargin - .23 in

\begin{document}


\title[Global regulairity for hyperbolic liquid crystal]
{Small data global regularity for 3-D Ericksen-Leslie's hyperbolic liquid crystal model without kinematic transport}
\author[J. Huang]{Jiaxi Huang$^1$}
\email{\href{mailto:huangjiaxi@bicmr.pku.edu.cn}{huangjiaxi@bicmr.pku.edu.cn}}

\author[N. Jiang]{Ning Jiang$^{2}$}
\email{\href{mailto:njiang@whu.edu.cn}{njiang@whu.edu.cn}}

\author[Y. Luo]{Yi-Long Luo$^3$}
\email{\href{mailto:yl-luo@whu.edu.cn}{yl-luo@whu.edu.cn}}

\author[L. Zhao]{Lifeng Zhao$^4$}
\email{\href{mailto:zhaolf@ustc.edu.cn}{zhaolf@ustc.edu.cn}}

\address{$^1$ Beijing International Center for Mathematical Research, Peking University, Beijing 100871, P. R. China}

\address{$^2$School of Mathemtical and Statistics, Wuhan University, Wuhan 430072, P. R. China}

\address{$^3$School of Mathematics, South China University of Technology, Guangzhou, 510641, P. R. China}

\address{$^4$School of Mathemtical Sciences, University of Science and Technology of China, Hefei 230026, P. R. China}

\subjclass[2010]{}

\keywords{Global regularity, hyperbolic, liquid crystal}

\begin{abstract}
In this article, we consider the Ericksen-Leslie's hyperbolic system for incompressible liquid crystal model without kinematic transport in three spatial dimensions, which is a nonlinear coupling of incompressible Navier-Stokes equations with wave map to $\mathbb{S}^2$.  Global regularity for small and smooth initial data near the equilibrium is proved.  The proof relies on the idea of space-time resonance. 
\end{abstract}
\maketitle

\setcounter{tocdepth}{2}

\section{Introduction}
The hydrodynamic theory of incompressible liquid crystals was established by Ericksen \cite{Ericksen-1961-TSR, Ericksen-1987-RM, Ericksen-1990-ARMA} and Leslie \cite{Leslie-1968-ARMA, Leslie-1979} in the 1960's (see also Section 5.1 of \cite{Lin-Liu-2001} ). The general Ericksen-Leslie's system consists of the following equations of the velocity field $u(x,t)\in \mathbb{R}^3$ and the orientation field $d(x,t)\in \mathbb{S}^2$, and $(x,t)\in \R^3\times \R^+$:
\begin{equation}\label{PHLC}
\begin{aligned}
\left\{ \begin{array}{c}
\partial_t u + u \cdot \nabla u  + \nabla p = - \div (\nabla d \odot \nabla d) + \div \tilde{\sigma}\,, \\
\div u = 0\, ,\\
\rho_1 \ddot{d} = \Delta d + \Gamma d + \lambda_1 (\dot{d} + B d) + \lambda_2 A d\,.
\end{array}\right.
\end{aligned}
\end{equation}
For the detailed derivation of \eqref{PHLC} from the original form of Ericksen-Leslie's formulation, see, for example \cite{Jiang-Luo-2018}.

In the above system, $\rho_1 > 0$ is the inertial constant, and the superposed dot denotes the material derivative $\partial_t + u \cdot\nabla$, and
\begin{equation*}
\begin{aligned}
A = \tfrac{1}{2}(\nabla u + \nabla^\top u)\,,\quad B= \tfrac{1}{2}(\nabla u - \nabla^\top u)\,,\end{aligned}
\end{equation*}
represent the rate of strain tensor and skew-symmetric part of the strain rate, respectively. We also define $N = \dot d + B d$ as the rigid rotation part of director changing rate by fluid vorticity. Here $A_{ij} = \tfrac{1}{2} (\partial_j u_i + \partial_i u_j)$, $ B_{ij} = \tfrac{1}{2} (\partial_j u_i - \partial_i u_j) $, $(B d)_i =B_{ki} d_k$, and $(\nabla d \odot \nabla d)_{ij} = \partial_i d_k \partial_j d_k$. The stress tensor $\tilde\sigma$ has the following form:
\begin{equation}\label{Extra-Sress-sigma}
\begin{aligned}
\tilde\sigma_{ji}=  \nu_1 d_k A_{kp}d_p  d_i d_j + \nu_2  d_j N_i  + \nu_3 d_i N_j  + \nu_4 A_{ij} + \nu_5 A_{ik}d_k d_j   + \nu_6 d_i A_{jk}d_k \,.
\end{aligned}
\end{equation}
These coefficients $\nu_i (1 \leq i \leq 6)$ which may depend on material and temperature, are usually called Leslie coefficients, and are related to certain local correlations in the fluid. Usually, the following relations are frequently introduced in the literatures \cite{Ericksen-1961-TSR,Leslie-1968-ARMA, Wu-Xu-Liu-ARMA2013}.
\begin{equation}\label{Coefficients-Relations}
\lambda_1=\nu_2-\nu_3\,, \quad\lambda_2 = \nu_5-\nu_6\,,\quad \nu_2+\nu_3 = \nu_6-\nu_5\,.
\end{equation}
The first two relations are necessary conditions in order to satisfy the equation of motion identically, while the third relation is called {\em Parodi's relation}, which is derived from Onsager reciprocal relations expressing the equality of certain relations between flows and forces in thermodynamic systems out of equilibrium. Under Parodi's relation, we see that the dynamics of an incompressible nematic liquid crystal flow involve five independent Leslie coefficients in \eqref{Extra-Sress-sigma}. Furthermore, in \eqref{PHLC}, the Lagrangian multiplier $\Gamma$ is (which ensures the geometric constraint $|d|=1$):
\begin{equation}\label{Lagrange-Multiplier}
\Gamma = - \rho_1 |\dot{d}|^2 + |\nabla d|^2 - \lambda_2 d^\top A d\, .
\end{equation}
We remark that the last two terms in the third equation of \eqref{PHLC} is the so-called kinematic transport, i.e.
\begin{equation}\label{kinematic transport}
  g= \lambda_1 (\dot{d} + B d) + \lambda_2 A d\,,
\end{equation}
which represents the effect of the macroscopic flow field on the microscopic structure. The material coefficients $\lambda_1$ and $\lambda_2$ reflect the molecular shape and the slippery part between the fluid and particles. The first term represents the rigid rotation of the molecule , while the second term stands for the streching of the molecule by the flow. 

Recently, the second and third named authors of the current paper studied in \cite{Jiang-Luo-2018} the well-posedness in the context of classical solutions for the hyperbolic case, i.e. $\rho_1 >0$. More precisely, in \cite{Jiang-Luo-2018} under some natural constraints on the Leslie coefficients which ensure the basic energy law is dissipative, they proved the local-in-time existence and uniqueness of the classical solution to the system \eqref{PHLC} with finite initial energy. Furthermore, with an additional assumption on the coefficients which provides a damping effect, i.e. $\lambda_1 < 0$, and the smallness of the initial energy, the unique global classical solution was established. Here we remark that the assumption $\lambda_1 < 0$ plays a crucial role in the global-in-time well-posedness.

 Cai-Wang \cite{CW} recently made progress for the simplied Ericksen-Leslie system, namely, the case with  $\nu_i = 0, i=1\,,\cdots\,,6$, $i\neq 4$. They proved the global regularity of \eqref{PHLC} near the constant equilibrium by employing the vector field method.

In the current paper, we consider the more general case: still $\nu_2=\nu_3=0$ and $\nu_4>0$, but $\nu_5=\nu_6 >-\nu_4$, and $\nu_1 > -2(\nu_4+\nu_5)$. Of course, from \eqref{Coefficients-Relations}, we still have $\lambda_1=\lambda_2=0$, i.e. the kinematic transport vanishes. The aim of this paper is to prove the global regularity of this more general case (but still special, comparing to the most general case, say, \eqref{Extra-Sress-sigma} and \eqref{Coefficients-Relations}) near the constant equilibrium $(u,d)= (\vec{0},\vec{i})$. More precisely, we study the Ericksen-Leslie's hyperbolic liquid crystal model in the following form: 
\begin{equation}        \label{ori_sys}
\left\{
\begin{aligned}
\partial_t u+u\cdot\nabla u+\nabla p & =\tfrac{\nu_4}{2}\Delta u-\div(\nabla d \odot \nabla d)+\div \sigma, \\
\div u & =0,\\
\ddot{d}-\Delta d & =(-|\dot{d}|^2+|\nabla d|^2)d.
\end{aligned}
\right.
\end{equation}
on $\R^3\times \R^+$ with the constraint $|d|=1$, where
\begin{equation*}
\sigma_{ji}=\nu_1 d_kd_pA_{kp}d_id_j+ \nu_5(d_j d_k A_{ki}+d_i d_k A_{kj})\,.
\end{equation*}

\subsection{Historical remarks}
Important special case of \eqref{PHLC} $\rho_1 = 0, \lambda_1=-1$ is the so-called parabolic Ericksen-Leslie system, which has been extensively studied from the mid 80's.The static analogue of the parabolic Ericksen-Leslie's system is the so-called Oseen-Frank model, whose mathematical study was initialed from Hardt-Kinderlehrer-Lin \cite{Hardt-Kinderlehrer-Lin-CMP1986}. Since then there have been many works in this direction. In particular, the existence and regularity or partial regularity of the approximation (usually Ginzburg-Landau approximation as in \cite{Lin-Liu-CPAM1995}) dynamical Ericksen-Leslie's system was started by the work of Lin and Liu in \cite{Lin-Liu-CPAM1995}, \cite{Lin-Liu-DCDS1996} and \cite{Lin-Liu-ARMA2000}. The simplest system preserving the basic energy law which can be obtained by neglecting the Leslie stress and by specifying some elastic constants. In 2-D case, the existence of global weak solutions with at most a finite number of singular times were proved by Lin-Lin-Wang \cite{Lin-Lin-Wang-ARMA2010}. Recently, Lin and Wang proved global existence of weak solution for 3-D case with the initial director field lying in the hemisphere in \cite{Lin-Wang-CPAM2016}.

For the more general parabolic Ericksen-Leslie's system,  local well-posedness is proved by Wang-Zhang-Zhang in \cite{Wang-Zhang-Zhang-ARMA2013}, and  existence of global solutions and regularity in $\mathbb{R}^2$ was established by Huang-Lin-Wang in \cite{Huang-Lin-Wang-CMP2014}. For more complete review of the works for the parabolic Ericksen-Leslie's system, please see the references therein above.

If $\rho_1>0$,  \eqref{PHLC} is an incompressible Navier-Stokes equations coupled with a wave map type system for which there is very few works, comparing the corresponding parabolic model, which is Navier-Stokes coupled with a heat flow. The only notable exception might be for the most simplified model, say, in \eqref{PHLC}, taking $\tilde{\sigma}=0$ and $u=0$, and the spatial dimension is $1$. For this case, the system \eqref{PHLC} can be reduced to a so-called nonlinear variational wave equation. Zhang and Zheng  studied systematically the dissipative and energy conservative solutions \cite{Zhang-Zheng-CPDE2001,Zhang-Zheng-ARMA2010}.

For the multidimensional case, there has been some progress on the hyperbolic system of liquid crystal. Very recently, De Anna and Zarnescu \cite{DeAnna-Zarnescu-2016} considered the inertial Qian-Sheng model of liquid crystals. They derived the energy law and proved the local well-posdedness for bounded initial data and global well-posedness under the assumptions that the initial data is small in suitable norm and the coefficients satisfy some further damping property. Furthermore, for the inviscid version of the Qian-Sheng model, in \cite{FRSZ-2016}, Feireisl-Rocca-Schimperna-Zarnescu proved the global existence of the {\em dissipative solution} which is inspired from that of incompressible Euler equation defined by P-L. Lions \cite{Lions-1996}.

\subsection{The main theorem}
To state our main theorem, we need some notations. Define the perturbed angular momentum operators by
\begin{equation*}
\tilde{\Om}_i u=\Om_i u+A_i u,\ \tilde{\Om}_i d=\Om_i d,
\end{equation*}
where $\Om=(\Om_1,\Om_2,\Om_3)$ is the rotation vector-field $\Om=x\wedge\nab$ and $A_i$ is defined by
\begin{equation*}
A_1=\left(\begin{array}{ccc}
0&0&0\\
0&0&1\\
0&-1&0
\end{array}\right),\ \ 
A_2=\left(\begin{array}{ccc}
0&0&-1\\
0&0&0\\
1&0&0
\end{array}\right),\ \ 
A_3=\left(\begin{array}{ccc}
0&1&0\\
-1&0&0\\
0&0&0
\end{array}\right).
\end{equation*}
We define the scaling vector-field $S$ by
\begin{equation*}
S=t\d_t+x_i\d_{x_i}.
\end{equation*}

Let
\begin{equation*}
\Gamma\in \{\d_t,\d_1,\d_2,\d_3,\tilde{\Om}_1,\tilde{\Om}_2,\tilde{\Om}_3\}
\end{equation*}
and $Z^a=S^{a_1}\Gamma^{a'}$, where $a=(a_1,a'):=(a_1,a_2,\cdots,a_8)\in\Z_+^8$, $\Gamma^{a'}=\Gamma^{a_2}\Gamma^{a_3}\cdots\Gamma^{a_8}$, we define
\begin{equation*}
u^{(a)}:=Z^a u,\ \ d^{(a)}:=Z^a d.
\end{equation*}

The main result of this paper is as follows:
\begin{theorem}           \label{Ori_thm}
	Assume that $N_0:=60,N_1:=6,h:=6$, the fixed coefficients $\nu_1,\nu_4$ and $\nu_5$ satisfy
	\begin{equation}       \label{mu1mu2}
	\nu_4>0,\ \nu_1>-2(\nu_4+\nu_5),\ \nu_5>-\nu_4,
	\end{equation}
	and $(u_0,d_0,d_1)$ are initial data near equilibrium $(\vec{0},\vec{i},\vec{0})$ satisfying the smallness assumptions
	\begin{equation}            \label{MainAss_dini}
	\sup_{|a|\leq N_1} \{\lV  u_0^{(a)}\rV_{H^{N(a)}}+\lV \nab d_0^{(a)}\rV_{H^{N(a)}}+\lV d_1^{(a)}\rV_{H^{N(a)}}\}\leq \epsilon_0,
	\end{equation}
	where $N(a)=N_0-|a|h$ for $0\leq|a|\leq N_1$. Then there exists a unique global solution $(u,d)$ of the system (\ref{ori_sys}) with initial data
	\begin{equation*}
	u(0)=u_0,\ \ d(0)=d_0,\ \d_t d(0)=d_1,
	\end{equation*}
	satisfies the energy bounds
	\begin{gather*}
	\sup_{0\leq t\leq T}\sup_{|a|\leq N_1} \{\lV  u^{(a)}(t)\rV_{H^{N(a)}}+ \lV\nab u^{(a)}\rV_{L^2([0,t]:H^{N(a)})}\}+\sup_{0\leq t\leq T}\lV \nab_{t,x} d(t)\rV_{H^{N(0)}}\lesssim \ep_0,\\
	\sup_{0\leq t\leq T}\sup_{1\leq|a|\leq N_1}\lV \nab_{t,x} d^{(a)}(t)\rV_{H^{N(a)}}\lesssim \ep_0(1+T)^{\bar{\delta}}.
	\end{gather*}
	for any $T\in[0,\infty)$, where $\bar{\delta}<10^{-7}$ depends on $\ep_0$, $N_0, N_1,\nu_1,\nu_4$, and $\nu_5$.
\end{theorem}

\begin{remark}
(i) Even though only the case $\nu_1$, $\nu_5\geq0$ was considered in \cite{Jiang-Luo-2018}, their local wellposedness results for small data hold true when $\nu_1$ or $\nu_5$ is negative. 

(ii)  We don't intend to get optimal regularity, hence we choose $N_0,N_1$ and $h$ sufficiently large just for convenience. In fact, the regularity index may be lowered.
\end{remark}

To understand the stress tensor $\div \sigma$ and the term $(-|\dot{d}|^2+|\nab d |^2)d$ more clearly, we introduce the polar coordinates $\phi_1,\phi_2$. Precisely, $\phi_1$ represents the angle between $x$-axis and projection of $d$ onto $x-y$ plane and $\phi_2$ represents the angle between $d$ and $x-y$ plane, namely
\begin{equation}          \label{d_pres}
d=(\cos\phi_1\cos\phi_2,\sin\phi_1\cos\phi_2,\sin\phi_2).
\end{equation}
Then by the orientation $d$ near $\vec{i}$, we know that the  angles $\phi_1,\phi_2$ are near $0$. Now the systerm (\ref{ori_sys}) can be rewritten as the $(u,\phi)$-system:
\begin{equation}                   \label{Main_Sys}
\left\{\begin{aligned}
\partial_t u+\tilde{L} u   & =-u\cdot\nabla u-\nab p-\d_j(\nab \phi\cdot \d_j \phi)+\nab\otimes(\phi\otimes\nab u)+Err11+Err12, \\
\div u & =0,\\
\d_t^2 \phi-\Delta\phi&=-\d_tu\cdot\nab \phi-2u\cdot\nab\d_t\phi-u\cdot\nab(u\cdot\nab\phi)+Err2
\end{aligned}
\right.
\end{equation}
where the second order linear operator $\tilde{L}$ is defined by
\begin{equation*}
\tilde{L}u:=-\frac{\nu_4}{2}\Delta u-(\frac{\nu_5}{2}\Delta u_1+(\nu_1+\nu_5)\d_1^2 u_1,\frac{\nu_5}{2}(\d_1^2 u_2+\d_1\d_2 u_1),\frac{\nu_5}{2}(\d_1^2 u_3+\d_1\d_3 u_1))^{\top},
\end{equation*}
the quadratic terms are
\begin{align*}
\nab\otimes(\phi\otimes\nab u):=&\nu_1\left(\begin{array}{ccc}
\d_2(A_{11}\phi_1)+\d_3(A_{11}\phi_2)+2\d_1(\phi_1A_{12}+\phi_2A_{13})\\
\d_1(A_{11}\phi_1)\\
\d_1(A_{11}\phi_2)
\end{array}\right)\\
&+\nu_5\left(\begin{array}{ccc}
\d_2(\phi_1A_{11})+\d_3(\phi_2A_{11})+\d_1(\phi_1A_{21}+\phi_2A_{31})+\d_j(\phi_1A_{2j}+\phi_2A_{3j})\\
\d_2(\phi_1A_{12})+\d_3(\phi_2A_{12})+\d_1(\phi_1A_{22}+\phi_2A_{32})+\d_j(\phi_1A_{1j})\\
\d_2(\phi_1A_{13})+\d_3(\phi_2A_{13})+\d_1(\phi_1A_{23}+\phi_2A_{33})+\d_j(\phi_2A_{1j})
\end{array}\right),
\end{align*}
and the error terms are 
\begin{gather*}
Err11=(Err11_1,Err11_2,Err11_3)^{\top},\ Err12:=\d_j(\sin^2 \phi_2\nab\phi_1 \d_j\phi_1),\\
Err2:=(2\tan\phi_2(\dot{\phi_1}\dot{\phi_2}-\nab\phi_1\nab\phi_2),\frac{1}{2}\sin 2\phi_2(-|\dot{\phi_1}|^2+|\nab\phi_1|^2)^{\top},
\end{gather*}
where
\begin{align*}
Err11_i=&\nu_1\d_j (A_{kp}\int_0^1 (\phi_1\d_1+\phi_2\d_2)^2(d_kd_pd_i d_j)(s\phi_1,s\phi_2)(1-s)ds)\\
&+\nu_5\d_j(A_{ki}\int_0^1 (\phi_1\d_1+\phi_2\d_2)^2(d_jd_k)(s\phi_1,s\phi_2)(1-s)ds)\\
&+\nu_5\d_j(A_{kj}\int_0^1 (\phi_1\d_1+\phi_2\d_2)^2(d_id_k)(s\phi_1,s\phi_2)(1-s)ds).
\end{align*}
We refer the readers to the appendix for details. 

Now it suffices to consider the system (\ref{Main_Sys}). For simplicity, we denote
\[
\phi^{(a)}:=Z^{a}\phi.
\] 
We establish the following result:
\begin{theorem}\label{Main_thm}
	With the notation and hypothesis in Theorem \ref{Ori_thm}, the initial data $(u_0,\phi_{j,0},\phi_{j,1})$ satisfies
	\begin{equation}            \label{MainAss_ini}
	\sup_{|a|\leq N_1} \big\{\lV  u_0^{(a)}\rV_{H^{N(a)}}+\sum_{j=1}^2(\lV \nab \phi_{j,0}^{(a)}\rV_{H^{N(a)}}+\lV \phi_{j,1}^{(a)}\rV_{H^{N(a)}})\big\}\leq \epsilon_0.
	\end{equation}
	Then there exists a unique global solution $(u,\phi)$ of the system (\ref{Main_Sys}) with initial data
	\begin{equation*}
	u(0)=u_0,\ \ \phi_j(0)=\phi_{j,0},\ \d_t \phi_j(0)=\phi_{j,1}, {\rm for}\ j=1,2,
	\end{equation*}
	satisfies the energy bounds
	\begin{gather*}
	\sup_{0\leq t\leq T}\sup_{|a|\leq N_1} \{\lV  u^{(a)}(t)\rV_{H^{N(a)}}+ \lV\nab u^{(a)}\rV_{L^2([0,t]:H^{N(a)})} \}+\sup_{0\leq t\leq T}\lV \nab_{t,x} \phi(t)\rV_{H^{N(0)}} \lesssim \ep_0,\\\label{Bdphia}
	\sup_{0\leq t\leq T}\sup_{1\leq|a|\leq N_1}\lV \nab_{t,x} \phi^{(a)}(t)\rV_{H^{N(a)}}\lesssim \ep_0(1+T)^{\bar{\delta}}.
	\end{gather*}
	for any $T\in[0,\infty)$, where $\bar{\delta}<10^{-7}$ depends on $\ep_0$, $N_0, N_1,\nu_1,\nu_4$, and $\nu_5$.
\end{theorem}

\subsection{Main ideas} 
The main strategy to prove global regularity relies on an interplay between the control of high order energies and decay estimates, which is based on the method of space-time resonances developed by Germain, Masmoudi and Shatah \cite{GMS08,GMS11,GMS12}. The main ingredients include decay estimates, energy and weighted energy estimates and $L^2$-bounds on the derivatives of profile $\widehat{\Psi}=\FF(e^{-it|\nab|}\Phi)$ associated with normalized solution $\Phi=\d_t\phi+i|\nab|\phi$. However, there are still some difficulties to get around.

\

In the proof of the Theorem \ref{Main_thm}, we need various decay estimates of $u$. However, for the decay of $u$, the presence of linear operator $\tilde{L}$ coming from $\div \sigma$ brings the first difficulty. In order to get around the difficulty, we introduce the vector $v$ and diagonalize the $u$-equation, i.e 
\begin{equation}          \label{U}
v=\mathbb{U}u:=(u_1,\tfrac{-i\d_2}{\sqrt{-\d_2^2-\d_3^2}}u_2+\tfrac{-i\d_3}{\sqrt{-\d_2^2-\d_3^2}}u_3,\tfrac{-i\d_3}{\sqrt{-\d_2^2-\d_3^2}}u_2+\tfrac{i\d_2}{\sqrt{-\d_2^2-\d_3^2}}u_3).
\end{equation}
Then the system (\ref{Main_Sys}) can be further rewritten as 
\begin{equation}                   \label{Main_Sys-v}
\left\{\begin{aligned}
&\partial_t v+Lv  =\mathbb{U}\mathbb{P}(-\U v\cdot\nabla \U v-\sum_{j=1}^3\d_j(\nab \phi \d_j \phi )+\nab\otimes(\phi\otimes\nab u)+Err11+Err12),\\
&\div v  =0,\\
&\d_t^2 \phi-\Delta\phi=-\d_tu\cdot\nab \phi-2u\cdot\nab\d_t\phi-u\cdot\nab(u\cdot\nab\phi)+Err2
\end{aligned}
\right.
\end{equation}
where $\P$ is the Leray projection and $L$ is the operator
\begin{equation}         \label{L-def}   
L v:=\Big(-\frac{\nu_4+\nu_5}{2}\Delta v_1+\nu_1\frac{\d_1^2(\d_2^2+\d_3^2)}{|\nab|^2}v_1,-\frac{\nu_4+\nu_5}{2}\Delta v_2+\nu_1\frac{\d_1^2(\d_2^2+\d_3^2)}{|\nab|^2}v_2,(-\frac{\nu_4}{2}\Delta-\frac{\nu_5}{2}\d_1^2) v_3\Big)^{\top}.
\end{equation}
Thus we can obtain the linear decay estimates of $e^{-tL}$, which will be used for the decay of $v$.

\begin{remark}
The coefficients condition (\ref{mu1mu2}) is used to ensure that the above $v$-equation is a parabolic equation. Note that the system (\ref{Main_Sys}) is a quasilinear parabolic-hyperbolic system with dissipation and dispersion effects. Moreover, it is the dissipation of $v$ that ensures the boundedness of the higher order energies of $v$ which will be used for the higher order energy estimates of $\phi$. In fact, the global existence of the system (\ref{Main_Sys}) will be tremendously difficult without dissipation effect.
\end{remark}

\

In the proof of the decay of $v$, the quadratic term $\nab\otimes(\phi\otimes\nab u)$ coming from $\div\sigma$ brings the second difficulty. In fact, due to the presence of the quadratic term, the argument for the decay of $v$ is not closed if we only use the decay of heat semi-group $e^{-tL}$ and bootstrap assumptions. Here, we have to exploit the decay of $\phi$ in $L^{\infty}$ and $\nabla^2u$ in $L^2$.

\

In proving the energy and weighted energy estimates and $L^2$-bounds on the derivatives of profile $\widehat{\Psi}$, the quadratic term $u\cdot\nab\d_t\phi$ which comes from the second order material derivative term $(\d_t+u\cdot\nab)^2\phi$ brings the third difficulty. More precisely, since the decays of $u$ and $\nab\phi$ are not good enough, we ultilize the structure of the quadratic term to gain more decay. To get around the difficulty, we define the normalized solution $\Phi$ and its associated profile by
\begin{equation*}
\Phi:=\d_t\phi+i|\nab|\phi,\ \ \Psi:=e^{-it|\nab|}\Phi.
\end{equation*}
We write Duhamel's formula in Fourier space for the profile $\Psi$ as follows:
\begin{equation*}
\widehat{\Psi}(t,\xi)\approx\widehat{\Phi}_0(\xi)+\sum_{\nu\in\{+,-\}}\int_0^t \int_{\R^3}e^{-is(|\xi|-\nu|\eta|)} \widehat{u}(\xi-\eta)i\eta \widehat{\Psi}_{\nu}(\eta)d\eta ds,
\end{equation*}
where $\Psi_+=\Psi,\Psi_-=\bar{\Psi}$. Then from the phase $|\xi|-\nu|\eta|$ we know that the time resonant sets are
\begin{equation*}
\mathscr{T}_+=\{(\xi,\eta):|\xi|-|\eta|=0\},\ \ \mathscr{T}_-=\{(\xi,\eta):|\xi|+|\eta|=0\}=\{(0,0)\},
\end{equation*}
and space resonant set is 
\begin{equation*}
\mathscr{S}_{\pm}=\{(\xi,\eta):\nab_{\eta}(|\xi|\mp|\eta|)=0\}=\emptyset.
\end{equation*}
From this we obtain the space-time resonant set 
\begin{equation*}
\mathscr{R}_{\pm}=\mathscr{T}_{\pm}\cap \mathscr{S}_{\pm}=\emptyset.
\end{equation*}

Then, for weighted energy estimates, by $\phi$-equation one needs to consider the space-time integrals
\begin{equation}          \label{HardTerm}
\int_0^t\int_{\R^6}e^{-is(|\xi|-\nu|\eta|)}m(\xi,\eta) \overline{ \widehat{\Psi^{(a)}}}(\xi) \widehat{u^{(b)}}(\xi-\eta)\widehat{\Psi_{\nu}^{(c)}}(\eta)d\xi d\eta ds,
\end{equation}
Hence, by the above resonance analysis, for the contribution of high-low interaction, i.e $|\xi|\approx |\xi-\eta|\gg |\eta|$ and high-high interaction, i.e $|\xi|\ll|\xi-\eta|\approx|\eta|$, we may integrate by parts in time and then use the decay of $\d_s u^{(b)}$ and $\d_s\Psi^{(c)}$ to control (\ref{HardTerm}). For the contribution of low-high interaction, i.e $|\xi|\approx |\eta|>|\xi-\eta|$, since the space resonant set $\mathscr{S}$ is null, we may use integration by parts in $\eta$ and the decay of $\lV\nab u^{(b)}\rV_{L^2}$ and $\Phi^{(c)}$ to control the increment of (\ref{HardTerm})

For the $L^2$-bounds on the derivatives of $\widehat{\Psi^{(a)}}$, by bootstrap assumptions and the following observation
\begin{equation*}
\lV \FF^{-1}(|\xi|\d_{\xi}\widehat{\Psi^{(a)}})\rV_{H^N}\lesssim \lV S\Psi^{(a)}\rV_{H^N}+\lV \Om\Psi^{(a)}\rV_{H^N}+\lV \Psi^{(a)}\rV_{H^N}+\lV t\d_t\Psi^{(a)}\rV_{H^N},
\end{equation*}
it suffices to prove 
\begin{equation*}
\lV \d_t\Psi^{(a)}\rV_{H^N}\lesssim \ep_0 t^{-1+\kappa(a)\delta}.
\end{equation*}
Since the space-resonant set is empty, we may use integration by parts in $\eta$ to gain the estimate.

\subsection{Notations} The set of all Schwartz functions is called a Schwartz space and is denoted $\mathcal{S}(\R^d)$. For any $k\in\Z$,  let $k^+:=\max(k,0)$ and $k^-:=\min(k,0)$. For any $x\in\R$, denote $\<x\>:=\sqrt{1+x^2}$. For any two numbers $A$ and $B$ and a absolute constant $C$, we denote
\begin{equation*}
A\lesssim B,\ B\gtrsim A,\ \mathrm{if}\ A\leq CB.
\end{equation*}

Here we will use the following multi-index
\begin{equation}         \label{n&n0}
n:=(\alpha,\beta,\gamma)\in\N^3, {\rm and}\  n_0\in\{(1,0,0),(0,1,0),(0,0,1)\}.
\end{equation}
Denote
$$|n|:=\alpha+\beta+\gamma,\ \ \d^n:=\d^{\alpha}_{x_1}\d^{\beta}_{x_2}\d^{\gamma}_{x_3}.$$

\subsection{Outline}
In section 2, we fix notations and state the main bootstrap proposition. We also state several lemmas, such as dispersive linear bounds, Hardy-type estimate and weighted $L^{\infty}-L^2$ estimate. In section 3, we prove the linear decay estimate first, which is the key lemma to derive the decay estimate of $\d \phi$. Then using the bootstrap assumptions to derive various decay estimates of $v=\U u$ and $\d \phi$.

We then start the proof of the main bootstrap proposition in section 4 and 5, where we obtain improved energy estimates (\ref{Main_Prop_result1}) and the $L^2$ bounds (\ref{Main_Prop_result2}) on the derivation of $\widehat{\Psi}$.

\section{Preliminaries and the main propositions}

In this section, we start by summarizing our main definitions and notations.

\subsection{Some analysis tools}
The Fourier transform of $f$ is defined as follows:
\begin{equation*}
\mathcal{F}(f)(\xi):=\widehat{f}(\xi)=\int_{\R^3}e^{-ix\cdot\xi}f(x)dx.
\end{equation*}
We use $\mathcal{F}^{-1}(f)$ to denote the inverse Fourier transform of $f$. Fix an even smooth function $\varphi:\R \rightarrow [0,1] $ supported in $[-2,2]$ and equal to 1 in $[-1,1]$. For simplicity of notation, we also let $\varphi:\R^3 \rightarrow [0,1]$ denote the corresponding radial function on $\R^3$. For any $k\in \Z$, $I\subset \R$, let
\begin{gather*}
\varphi_k(x):=\varphi(x/2^k)-\varphi(x/2^{k-1}),\ \ \varphi_I(x):=\sum_{m\in I\cap \Z}\varphi_m(x),\\
\varphi_{\leq k}(x):=\sum_{l\leq k}\varphi_{l}(x),\ \varphi_{> k}(x):=\sum_{l> k}\varphi_{l}(x).
\end{gather*}
The frequency projection operator $P_k$, $P_I$, $P_{\leq k}$ and $P_{>k}$ is defined by the Fourier multiplier $\varphi_k(\xi)$, $\varphi_I(\xi)$, $\varphi_{\leq k}(\xi)$ and $\varphi_{>k}(\xi)$, i.e.
\begin{gather*}
\widehat{P_k f}(\xi)=\varphi_k(\xi)\widehat{f}(\xi),\ \ \ \widehat{P_I f}(\xi)=\varphi_I(\xi)\widehat{f}(\xi),\\
\widehat{P_{\leq k} f}(\xi)=\varphi_{\leq k}(\xi)\widehat{f}(\xi)\ \mathrm{and} \ \widehat{P_{> k} f}(\xi)=\varphi_{> k}(\xi)\widehat{f}(\xi).
\end{gather*}
Moreover, we have the following Bernstein inequality: For any $k\in \Z$,
\begin{equation}        \label{BernIneq}
	\lV P_k f\rV_{L^q}\lesssim 2^{3k(\frac{1}{p}-\frac{1}{q})}\lV P_k f\rV_{L^p},\ \ \ 1\leq p\leq q\leq \infty.
\end{equation}

Let
\begin{equation*}
\mathcal{J}:=\{(k,j)\in\Z\times\Z_+:k+j\geq 0\}.
\end{equation*}
For any $(k,j)\in\mathcal{J}$, let
\begin{equation*}
\tilde{\varphi}_j^{(k)}(x):=\left\{
\begin{aligned}
&\varphi_{\leq -k}(x),\ \mathrm{if}\ k+j=0\ \mathrm{and}\ k\leq 0,\\
&\varphi_{\leq 0}(x),\ \mathrm{if}\ j=0\ \mathrm{and}\ k\geq 0,\\
&\varphi_{j}(x),\ \mathrm{if}\ k+j\geq 1\ \mathrm{and}\ j\geq 1,
\end{aligned}
\right.
\end{equation*}
then for any $k\in\Z$ fixed, $\sum_{j\geq \max(-k,0)}\tilde{\varphi}_j^{(k)}=1$. For any $(k,j)\in\mathcal{J}$, we define the operator $Q_{jk}$ by
\begin{equation*}
(Q_{jk}f)(x):=\tilde{\varphi}_j^{(k)}(x)\cdot P_kf(x),
\end{equation*}
and denote
\begin{equation*}
f_{j,k}(x):=P_{[k-2,k+2]}Q_{jk}f(x).
\end{equation*}
Then $P_k f$ for fix $k\in\Z$ can be decomposed as
\begin{equation}      \label{Decom_Pkf}
P_k f(x)=P_{[k-2,k+2]}P_kf(x)=P_{[k-2,k+2]}\sum_{j\geq \max(-k,0)}\tilde{\varphi}_j^{(k)}(x)P_kf(x)=\sum_{j\geq \max(-k,0)}f_{j,k}(x).
\end{equation}
Moreover, for any $f\in L^2(\R^3)$ and $(k,j)\in \mathcal{J}$, let $\beta:=1/1000$, we have
\begin{equation}    \label{Dec_fjk}
\lV \widehat{f_{j,k}}\rV_{\Lf}\lesssim \min\{2^{3j/2}\lV Q_{jk}f\rV_{L^2},\ 2^{j/2-k}2^{\beta(j+k)}\sum_{|\alpha|\leq 1}\lV Q_{jk}\Om^{\alpha}f\rV_{L^2}\},
\end{equation}
see \cite[Lemma 3.4.]{IoPa}.

We state several decay estimates and dispersive estimates.
\begin{lemma}[Decay estimates]
	(\romannumeral1) For any Schwartz function $f\in \mathcal{S}(\R^3)$, we have
	\begin{gather}        \label{Dec_heat}
	\lV e^{-tL}|\nab|^l f\rV_{\dot{W}^{N,q}}\lesssim t^{-\frac{3}{2}(\frac{1}{p}-\frac{1}{q})-\frac{l}{2}} \lV f\rV_{\dot{W}^{N,p}},\ \ \ \  1\leq p\leq q\leq \infty.
	\end{gather}
\end{lemma}
\begin{proof}
	By (\ref{mu1mu2}), we obtain
	\begin{equation*}         \label{H33}
	\frac{\nu_4}{2}|\xi|^2+\frac{\nu_5}{2}\xi_1^2=\frac{\nu_4+\nu_5}{2}\xi_1^2+\frac{\nu_4}{2}(\xi_2^2+\xi_3^2),
	\end{equation*}
	and
	\begin{equation*}       \label{H1122}
	\begin{aligned}
	\frac{\nu_4+\nu_5}{2}|\xi|^2+\nu_1 \frac{\xi_1^2}{|\xi|^2}(\xi_2^2+\xi_3^2)
	=& \frac{\nu_4+\nu_5}{2}|\xi|^2\Big(1+\frac{2\nu_1}{\nu_4+\nu_5}\big(\frac{\xi_1^2}{|\xi|^2}-\frac{\xi_1^4}{|\xi|^4}\big)\Big)\\
	\geq & \frac{\nu_4+\nu_5}{2}|\xi|^2\cdot\min\big\{1,1+\frac{\nu_1}{2(\nu_4+\nu_5)}\big\}.
	\end{aligned}
	\end{equation*}
	Then by the above two bounds, we apply a similar argument to decay estimates of heat operator to $e^{-tH}$ and then obtain the bound (\ref{Dec_heat}).
	
\end{proof}

We also need the following Hardy-type estimate involving localization in frequency and space.
\begin{lemma}[Lemma 3.5, \cite{IoPa}]       \label{AkBk_lem}
	For $f\in L^2(\R^3)$ and $k\in\Z$ let
	\begin{equation*}
	F_k(f):=\lV P_k f\rV_{L^2}+\sum\limits_{l=1}^3 \lV \varphi_k(\xi)\d_{\xi_l}\widehat{f}(\xi)\rV_{L^2},\ \ B_k(f):=\big[\sum\limits_{j\geq \max(-k,0)}2^{2j}\lV Q_{jk}f\rV_{L^2}^2\big]^{1/2}.
	\end{equation*}
	Then, for any $k\in\Z$,
	\begin{equation*}
	B_k\lesssim \left\{
	\begin{aligned}
	&\sum_{|k'-k|\leq 4}F_{k'},\ \mathrm{if}\ k\geq 0,\\
	&\sum_{k'\in\Z}F_{k'}2^{-|k-k'|/2}\min\{1,2^{k'-k}\},\ \mathrm{if}\ k\leq 0.
	\end{aligned}
	\right.
	\end{equation*}
\end{lemma}

\begin{lemma}[Lemma 2.2, \cite{Wxc}]    \label{ContSymb}
	If $f:\R^{6}\rightarrow\C$ and $k_1,k_2\in\Z$, then the following estimate holds:
	\begin{equation*}
	\lV \int_{\R^{6}} f(\xi_1,\xi_2) e^{i(x_1\cdot\xi_1+x_2\cdot\xi_2)}\varphi_{k_1}(\xi_1)\varphi_{k_2}(\xi_2)d\xi_1\xi_2\rV_{L^1_{x_1, x_2}}\lesssim \sum\limits_{m=0}^{7}\sum\limits_{j=1}^2 2^{mk_j}\lV\partial_{\xi_j}^m f\rV_{L^{\infty}}.
	\end{equation*}
\end{lemma}

To obtain the energy estimates, one often needs to analyze the symbols. Define a class of symbol as follows
\begin{equation*}
\mathcal{S}^{\infty}:=\{m:\R^6\rightarrow\C,\ m\ \text{is continous and } \lV\mathcal{F}^{-1}m\rV_{L^1}<\infty\},
\end{equation*}
whose associated norms are defined as
\begin{equation*}
\lV m\rV_{\mathcal{S}^{\infty}}:=\lV \mathcal{F}^{-1}m\rV_{L^1},
\end{equation*}
and
\begin{equation*}
\lV m\rV_{\mathcal{S}^{\infty}_{k,k_1,k_2}}:=\lV m(\xi,\eta)\varphi_k(\xi)\varphi_{k_1}(\xi-\eta)\varphi_{k_2}(\eta)\rV_{\mathcal{S^{\infty}}}.
\end{equation*}
Then we have

\begin{lemma}[Bilinear estimate, \cite{IoPu}]      \label{Sym_Pre_lemma}
	Given $m\in\mathcal{S}^{\infty}$ and two well-defined functions $f_1,\ f_2$, then the following estimate holds:
	\begin{gather*}
	\lV \mathcal{F}^{-1}\big(\int_{\R^3}m(\xi,\eta)\widehat{f}_1(\xi-\eta)\widehat{f}_2(\eta)d\eta\big)(x)\rV_{L^r}\lesssim\lV m\rV_{\mathcal{S}^{\infty}} \lV f_{L^1}\rV_{L^p} \lV f_2\rV_{L^q},\ \ \frac{1}{r}=\frac{1}{p}+\frac{1}{q}.
	\end{gather*}
\end{lemma}

By standard Littlewood-Paley decomposition and H\"{o}lder, we can obtain the following Hardy-type inequality.
\begin{lemma}[Hardy-type inequality]    \label{Hardy}
	For any $1< p< \infty$ and two well-defined functions $g_1$, $g_2$, then the following estimates holds:
	\begin{gather}    \label{Hardy_1}
	\lV g_1g_2\rV_{\dot{H}^N}\lesssim \sum_k 2^{Nk^+}\lV P_k g_1\rV_{L^{\infty}} \lV g_2\rV_{\dot{H}^N},\\   \label{Hardy_2}
	\lV g_1g_2\rV_{W^{N,p}}\lesssim \lV g_1\rV_{W^{N,p}} \lV g_2\rV_{L^{\infty}}+\lV g_1\rV_{L^{\infty}} \lV g_2\rV_{W^{N,p}}.
	\end{gather}
\end{lemma}

Finally, we need to record  the following weighted $L^{\infty}-L^2$ estimate
\begin{lemma}[Lemma 3.3, \cite{Si}]Let $f\in H^2(\R^3)$, $r=|x|$, then there holds
	\begin{equation}   \label{rv}
	\lV \<r\>f\rV_{\Lf}\lesssim \sum_{|\alpha|\leq 1}\lV\d_r\tilde{\Om}^{\alpha}f\rV_{L^2}^{1/2}\sum_{|\alpha|\leq 2}\lV\tilde{\Om}^{\alpha}f\rV_{L^2}^{1/2},
	\end{equation}
	provided the right hand side is finite.
\end{lemma}

\begin{lemma}
	For any $N\geq 0$ and Schwartz function $f\in\mathcal{S}(\R^3)$, we have
	\begin{equation}           \label{xidxif}
	\lV\FF^{-1}(|\xi|\nab_{\xi}\widehat{f})\rV_{H^N}\lesssim \lV x\cdot \nab_x f\rV_{H^N}+\lV \Om f\rV_{H^N}+\lV f\rV_{H^N}.
	\end{equation}
\end{lemma}
\begin{proof}
	This bound is obtained by the relation
	\begin{equation*}
	|\xi|\nab_{\xi}=\frac{\xi}{|\xi|}(\xi\cdot\nab_{\xi})-\frac{\xi}{|\xi|}\wedge \Om(\xi),
	\end{equation*}
	where $\Om(\xi)=\xi\wedge\nab_{\xi}$.
\end{proof}

\subsection{The main bootstrap proposition}
By standard argument, after applying vector fields $Z$ to the system (\ref{Main_Sys}) and make a change of unknown $$v^{(a)}:=\U u^{(a)},$$ 
we can derive that
\begin{equation}         \label{Main_Sys_VecFie}
\left\{\begin{aligned}
\d_t v^{(a)}+L v^{(a)}+\U\P L_1(u)=&-\U\P\sum_{b+c=a}C_a^b[ u^{(b)}\cdot\nab u^{(c)}+ \d_j(\nab(S-1)^{b_1}\Gamma^{b'}\phi\cdot\d_j(S-1)^{c_1}\Gamma^{c'}\phi)]\\
&-\U\P(S+1)^{a_1}\Gamma^{a'}\d_j(\phi\otimes\nab u)+(S+1)^{a_1}\Gamma^{a'}( Err11+Err12)),\\
\div u^{(a)}=&0,\\
\d_t^2\phi^{(a)}-\Delta\phi^{(a)}=&-\sum_{b+c=a}C_a^b(\d_t u^{(b)}\cdot\nab\phi^{(c)}+2u^{(b)}\cdot\nab\d_t\phi^{(c)})\\
&-\sum_{b+c+e=a}C_a^{b,c}u^{(b)}\cdot\nab(u^{(c)}\cdot\nab\phi^{(e)})+(S+2)^{(a_1)}\Gamma^{(a')}Err2,
\end{aligned}
\right.
\end{equation}
where $L_1(u)$ is the lower order terms from the commutation between $S,\tilde{\Om}$ and $\d$,
\begin{equation*}
L_1(u):=-\sum_{0\leq i\leq a_1-1}\frac{\nu_4}{2} C_{a_1}^i (-1)^{a_1-i} \Delta S^i\Gamma^{a'}u- (L_{1,c_1}u^{(c_1)},L_{2,c_2}u^{(c_2)},L_{3,c_3}u^{(c_3)})^{\top},
\end{equation*}
and where $L_{j,c_j}$ for $1\leq j\leq 3$ are second order differential operator,
\begin{equation*}
C_a^b:=\frac{a!}{b!(a-b)!},\ C_a^{b,c}:=\frac{a!}{b!c!(a-b-c)!}.
\end{equation*}

To state the main proposition we review the normalized solution $\Phi$ and it's profile $\Psi$, i.e for $|a|\leq N_1$
\begin{equation}       \label{Psi-Phi-relat}
\Phi^{(a)}:=(\d_t+i|\nab|) \phi^{(a)},\ \ \Psi^{(a)}:=e^{-it|\nab|}\Phi^{(a)}.
\end{equation}
The function $\phi$ can be recovered from the normalized variable $\Phi$ by the formulas
\begin{equation*}
\phi=\tfrac{1}{2i}|\nab|^{-1}(\Phi-\overline{\Phi}).
\end{equation*}

Our main result is the following proposition:
\begin{proposition}   \label{Main_Prop}
    Assume that $(v,\phi)$ is a solution to (\ref{Main_Sys-v}) on some time interval $[0,T]$, $T\geq 1$ with initial data satisfying the assumptions (\ref{MainAss_ini}). Assume also that the solution satisfies the bootstrap hypothesis
    \begin{gather}           \label{Main_Prop_Ass1}
        \sup\limits_{|a|\leq N_1,t\in[0,T]} \{\lV  v^{(a)}\rV_{H^{N(a)}}+\lV \nab v^{(a)}\rV_{L^2([0,t]:H^{N(a)})} +\<t\>^{-\ka(a)\delta}\lV \Phi^{(a)}\rV_{H^{N(a)}}\}\leq \epsilon_1,   \\\label{Main_Prop_Ass2}
        \sup_{|a|\leq N_1-1,t\in[0,T]}\<t\>^{-\ka(|a|+1)\delta}\lV\FF^{-1}(|\xi|\d_{\xi}\widehat{\Psi^{(a)}})\rV_{H^{N(|a|+1)}} \leq \ep_1,
    \end{gather}
    where $\ep_1=\ep_0^{2/3}$, $\delta=10^{-10}$,
    \begin{gather*}
    \ka(0)=0;\ \ka(a)=1,\ \mathrm{for}\ 1\leq |a|\leq N_1-2;\ \ka(a)=2,\ \mathrm{for}\ N_1-1\leq|a|\leq N_1.\\
    N(a)=N_0-|a|h,\ N_0=60,\ h=6.
    \end{gather*}
    Then the following improved bounds hold
    \begin{gather}\label{Main_Prop_result1}
        \sup\limits_{|a|\leq N_1,t\in[0,T]} \{\lV  v^{(a)}\rV_{H^{N(a)}}+\lV \nab v^{(a)}\rV_{L^2([0,t]:H^{N(a)})} +\<t\>^{-\ka(a)\delta}\lV \Phi^{(a)}\rV_{H^{N(a)}}\}\lesssim \epsilon_0,   \\\label{Main_Prop_result2}
        \sup_{|a|\leq N_1-1,t\in[0,T]}\<t\>^{-\ka(|a|+1)\delta}\lV\FF^{-1}(|\xi|\d_{\xi}\widehat{\Psi^{(a)}})\rV_{H^{N(|a|+1)}}  \lesssim \ep_0.
    \end{gather}
\end{proposition}

The bounds (\ref{Main_Prop_Ass2}) and (\ref{Main_Prop_result2}) are our main $L^2$ bounds on the derivatives of the profile $\Psi^{(a)}$ in the Fourier space. They correspond to weighted bounds in the physical space which plays an important role in the energy estimates.

From the assumption (\ref{Main_Prop_Ass2}) and Lemma $\ref{AkBk_lem}$, we give the following useful bound.
\begin{corollary}
	With the notation and hypothesis in Proposition \ref{Main_Prop}, we have for any $|a|\leq N_1-1$,
	\begin{equation}       \label{Qjk_Psi}
	\Big(\sum_{(k,j)\in\mathcal{J}}2^{2N(|a|+1)k^++2k}2^{2j} \lV Q_{jk}\Psi^{(a)}\rV_{L^2}^2\Big)^{1/2} \lesssim \ep_1\<t\>^{\ka(|a|+1)\delta}.
	\end{equation}
\end{corollary}
\begin{proof}
	The corollary is an easy consequence of Lemma $\ref{AkBk_lem}$ and (\ref{Main_Prop_Ass2}). Indeed, if $k\geq 0$,
	\begin{align*}
	\Big(\sum_{(k,j)\in\mathcal{J};j\geq 0,k\geq 0}2^{2N(|a|+1)k^++2k}2^{2j} \lV Q_{jk}\Psi^{(a)}\rV_{L^2}^2\Big)^{1/2}  \lesssim & \lV 2^{N(|a|+1)k^++k}\sum_{|k'-k|\leq 4}F_{k'}(\Psi^{(a)})\rV_{l^2_k}\\
	\lesssim & \lV \Psi^{(a)}\rV_{H^{N(|a|+1)+1}}+\lV \FF^{-1}(|\xi|\d_{\xi}\widehat{\Psi^{(a)}})\rV_{H^{N(|a|+1)}}\\
	\lesssim & \ep_1\<t\>^{\ka(|a|+1)\delta}.
	\end{align*}
	If $k<0$, we have
	\begin{align*}
	&\Big(\sum_{(k,j)\in\mathcal{J};j\geq -k,k<0}2^{2N(|a|+1)k^++2k}2^{2j} \lV Q_{jk}\Psi^{(a)}\rV_{L^2}^2\Big)^{1/2} \\
	\lesssim &\lV 2^k \sum_{k'}F_{k'}(\Psi^{(a)})2^{-|k-k'|/2}\min(1,2^{k'-k})\rV_{l^2_k}\\
	\lesssim &\lV \sum_{k'>k} 2^{k'}F_{k'}(\Psi^{(a)})2^{3(k-k')/2}+\sum_{k'\leq k}2^{k'}F_{k'}(\Psi^{(a)})2^{(k'-k)/2}\rV_{l^2_k}\\
	\lesssim & \lV 2^k F_k(\Psi^{(a)}) \rV_{l^2_k}\\
	\lesssim & \ep_1\<t\>^{\ka(|a|+1)\delta}.
	\end{align*}
	This completes the proof of the Corollary.
\end{proof}

Once Proposition $\ref{Main_Prop}$ is proved, Theorem $\ref{Main_thm}$ follows directly from the standard continuity argument. The rest of this paper focuses on the proof of Proposition $\ref{Main_Prop}$. The key ingredients include Proposition $\ref{Prop_Ene_Sob}-\ref{Prop_Phia}$ and Proposition $\ref{Prop_Psia}$.

\section{Decay of velocity field and orientation field}
In this section, we give the various decay estimates of $v$ and $\Phi$, which will be useful in the energy estimates in the next sections. 

\subsection{Decay of $\Phi$.} In order for the decay estimates of $\Phi$, the following frequency localized linear dispersive estimate is necessary.

\begin{lemma}[Frequency localized linear decay estimate] For any $k\in \Z$ and Schwartz function $f\in\mathcal{S}(\R^3)$, we have
	\begin{equation}    \label{Dis}
	\lV e^{it|\nab|} P_k f\rV_{L^{\infty}}\lesssim \lV \widehat{P_kf}\rV_{L^{\infty}}(t^{-1}2^{2k}+t^{-1+\delta}2^{2k+\delta k})+\lV \nab_{\xi}\widehat{P_kf}\rV_{L^2}(t^{-1}2^{3k/2}+t^{-1+\delta}2^{3k/2+\delta k}).
	\end{equation}
\end{lemma}
\begin{proof}
	By the similar argument to Lemma 4.1 in \cite{Wxc}, the bound (\ref{Dis}) follows.
\end{proof}

Next, we use the dispersive estimates to give the decay of $\Phi$.
\begin{lemma}                \label{Phi_inf}
	With the notations and hypothesis in Proposition \ref{Main_Prop}, for any $k\in\Z$, $|a|\leq N_1-2$ and $t\in[0,T]$ we have
	\begin{gather}     \label{Dec_Phi}
	\lV P_k\Phi^{(a)}(t)\rV_{L^{\infty}} \lesssim \ep_1 \<t\>^{-1+\delta+\ka(|a|+2)\delta} 2^{k^-/2}2^{-N(|a|+1)k^+ +(h/2+2)k^+},\\ \label{Dec_Phi/nab}
	\lV|\nab|^{-1}\Phi^{(a)}(t)\rV_{L^{\infty}}\lesssim \ep_1\<t\>^{-1/2+2\delta}.
	\end{gather}
\end{lemma}
\begin{proof}
	On one hand, by (\ref{Dec_fjk}) and (\ref{Qjk_Psi}) when $k>0$,  we have
	\begin{equation*}    \label{Psi_k>0}
	\begin{aligned}
	\lV \widehat{P_k\Psi^{(a)}}\rV_{L^{\infty}}\lesssim &\sum_{j\leq hk^+} 2^{3j/2}\lV Q_{jk}\Psi^{(a)}\rV_{L^2}+\sum_{j>hk^+}2^{j/2-k+\beta(j+k)}\sum_{|\alpha|\leq 1}\lV Q_{jk}\tilde{\Om}^{\alpha}\Psi^{(a)}\rV_{L^2}\\
	\lesssim& 2^{hk^+/2}\ep_1\<t\>^{\ka(|a|+2)\delta}2^{-N(|a|+1)k^+-k}+2^{-hk^+/2}\ep_1\<t\>^{\ka(|a|+2)\delta}2^{-N(|a|+2)k^+-k}\\
	\lesssim & \ep_1\<t\>^{\ka(|a|+2)\delta} 2^{-N(|a|+1)k^++hk^+/2-k}.
	\end{aligned}
	\end{equation*}
	On the other hand, when $k\leq 0$ we have
	\begin{equation*} \label{Psi_k<0}
	\begin{aligned}
	\lV \widehat{P_k\Psi^{(a)}}\rV_{L^{\infty}}\lesssim &\sum_{j\geq -k}2^{j/2-k+\beta(j+k)}\sum_{|\alpha|\leq 1}\lV Q_{jk}\tilde{\Om}^{\alpha}\Psi^{(a)}\rV_{L^2}\\
	\lesssim& \ep_1\<t\>^{\ka(|a|+2)\delta} 2^{-3k/2}.
	\end{aligned}
	\end{equation*}
	Meanwhile, by (\ref{Main_Prop_Ass2}) and the definition of $P_k$, we also have
	\begin{equation*} \label{Psi_2}
	2^k\lV \nab_{\xi}\widehat{P_k\Psi^{(a)}}(\xi)\rV_{L^2}= \lV \frac{\xi}{|\xi|}\varphi_k'(\xi) \widehat{\Psi^{(a)}}(\xi)\rV_{L^2}+2^k\lV \varphi_k(\xi)\nab_{\xi}\widehat{\Psi^{(a)}}\rV_{L^2}\lesssim \ep_1 \<t\>^{\ka(|a|+1)\delta}2^{-N(|a|+1)k^+}.
	\end{equation*}
	Thus (\ref{Dec_Phi}) follows from (\ref{Dis}) and the above three bounds. 
	
	Next by (\ref{BernIneq}), H\"{o}lder inequality, (\ref{Main_Prop_Ass1}) and (\ref{Dec_Phi}), we have for $p\geq 2$
	\begin{align*}
	\lV|\nab|^{-1}\Phi^{(a)}(t)\rV_{L^{\infty}}&\lesssim \sum_k 2^{3k/p-k}\lV P_k \Phi^{(a)}(t) \rV_{L^p} \lesssim \sum_k 2^{3k/p-k}\lV P_k \Phi^{(a)}(t) \rV_{L^2}^{2/p} \lV P_k \Phi^{(a)}(t) \rV_{L^{\infty}}^{1-2/p}\\
	&\lesssim \ep_1 t^{(-1+\delta+\ka(|a|+2)\delta)(1-2/p)} \sum_k 2^{(2/p-1/2)k^-}2^{(1-2/p)(-N(|a|+2)+h/2+2)k^++k^+/p}.
	\end{align*}
	Choosing $p$ such that $2/p-1/2=\delta/4$, the bound (\ref{Dec_Phi/nab}) follows.
\end{proof}

As a consequence of Lemma \ref{Phi_inf}, we have
\begin{corollary}
	With the notations and hypothesis in Proposition \ref{Main_Prop}, we have
	\begin{equation}      \label{phi^2-Hna}
	\sum_{|b|+|c|\leq|a|}\lV\nab\phi^{(b)}\nab\phi^{(c)}\rV_{H^{N(a)}}\lesssim \ep_1^2\<t\>^{-1+4\delta}.
	\end{equation}
\end{corollary}
\begin{proof}
	When $|a|\geq N_1-1$, by symmetry we may assume that $|b|\leq |c|$, then the bound (\ref{phi^2-Hna}) is obtained by (\ref{Hardy_1}) and (\ref{Dec_Phi}). When $|a|\leq N_1-2$, the bound (\ref{phi^2-Hna}) is obtained by (\ref{Hardy_2}) and (\ref{Dec_Phi}).
\end{proof}

\subsection{Decay of $v$.} In order for the decay of $\nab^l v$, $l=0,1,2$, in a function space $X$, by $v$-equation in (\ref{Main_Sys-v}) and Duhamel's formula, it suffices to estimate
\begin{equation*}
\lV \nab^l v(t)\rV_X\lesssim \lV e^{-tL}\nab^l v_0\rV_X+\int_0^t \lV e^{-(t-s)L} \nab^l f(s) \rV_Xds,
\end{equation*}
where $f$ denotes the nonlinearities of $v$-equation in (\ref{Main_Sys-v}). Then we obtain the desired decay estimates by the linear decay estimates (\ref{Dec_heat}) and the following bound
\begin{equation*}
\int_0^t\lV e^{-(t-s)L} \nab^l f(s) \rV_Xds\lesssim \ep_1^2\<t\>^{-c}+\ep_1\sup_{s\in[t/2,t]}\lV \nab^l v(s)\rV_{X},\ {\rm for\ some }\ c>0,
\end{equation*}
which is the main part in the proof of the decay of $\nab^l v$.

Due to the presence of quadratic term $\nab (\phi\nab u)$, we begin with the decay of $\nab^2 v$ in $L^2$ space. 

\begin{lemma}    \label{d2v_Hn_lem}
	With the notations and hypothesis in Proposition \ref{Main_Prop}, for any $t\in[0,T]$, $|a|\leq N_1$, we have the decay estimate
	\begin{gather}        \label{nab3_v(b,c)_Dec}
	\lV\nab^2 v^{(a)}\rV_{H^{N(|a|+1)}}\lesssim \ep_1\<t\>^{-1/2+(3+|a|)\delta}.
	\end{gather}
	As a consequence, we have
	\begin{gather}     \label{Hntdtv(a)}
	\lV \d_t v^{(a)}(t)\rV_{H^{N(|a|+1)}}\lesssim \ep_1\<t\>^{-1/2+(3+|a|)\delta}.
	\end{gather}
	Furthermore, we obtain the estimate
	\begin{gather}           \label{nab_xi_v(b,c)_HN}
	\lV\FF^{-1}(|\xi|\nab_{\xi}\widehat{u^{(a)}})\rV_{H^{N(|a|+1)}} \lesssim \ep_1 \<t\>^{1/2+(3+|a|)\delta},\ \mathrm{for}\ |a|\leq N_1-1.
	\end{gather}
\end{lemma}

\begin{proof}
	\emph{Step 1: Prove the estimate (\ref{nab3_v(b,c)_Dec}). } We prove the bound by induction. Assume that
	\begin{equation}     \label{dec_d2vaL2_ass}
	\lV\nab^2 v^{(c)}\rV_{H^{N(|c|+1)}}\lesssim \ep_1 \<t\>^{-1/2+(3+|c|\delta)},\ \ \mathrm{for}\ |c|<|a|,
	\end{equation}
	By Duhamel's formula and (\ref{Main_Sys_VecFie}), it suffices to prove
	\begin{gather}\label{dec_d2vL2_1}
	\lV e^{-tL}\nab^2 v_0^{(a)}\rV_{H^{N(|a|+1)}}\lesssim \ep_0 t^{-1},         \\\label{dec_d2vaL2_1'}
	\int_0^t\lV  e^{-(t-s)L} \nab^4 u^{(c)}(s)\rV_{H^{N(|a|+1)}}ds\lesssim \ep_1 \<t\>^{-1/2+(3+|a|)\delta},\ {\rm for}\ |c|<|a|,\\
	\label{dec_d2vL2_2}
	\int_0^t\lV  e^{-(t-s)L} \nab^2 f^{(a)}(s)\rV_{H^{N(|a|+1)}}ds\lesssim \ep_1^2 \<t\>^{-1/2+(3+|a|)\delta}+\ep_1 \sup_{s\in[t/2,t]} \lV\nab^2 v^{(a)}(s)\rV_{H^{N(|a|+1)}}.
	\end{gather}
	where the nonlinearities $f^{(a)}$ is
	\begin{equation*}       
	f^{(a)}=\nab(\phi^{(b)}\nab u^{(c)})+u^{(b)}\cdot\nab u^{(c)}+\nab(\Phi^{(b)}\Phi^{(c)})+Z^{a}(Err11+Err12).
	\end{equation*}
	In fact, once (\ref{dec_d2vL2_1})-(\ref{dec_d2vL2_2}) hold, from Duhamel's formula we have
	\begin{align*}
	\sup_{t\in[1,T]} t^{1/2-(3+|a|)\delta}\lV\nab^2 v^{(a)}(t)\rV_{H^{N(|a|+1)}}\lesssim \ep_1^2 +\ep_1 \sup_{t\in[1,T]} t^{1/2-(3+|a|)\delta}\lV\nab^2 v^{(a)}(t)\rV_{H^{N(|a|+1)}},
	\end{align*}
	which implies the bound (\ref{nab3_v(b,c)_Dec}).
	
	Now we begin to prove the bounds (\ref{dec_d2vL2_1})-(\ref{dec_d2vL2_2}). (\ref{dec_d2vL2_1}) is a consequence of (\ref{MainAss_ini}) and (\ref{Dec_heat}). For the second bound (\ref{dec_d2vaL2_1'}), using (\ref{Dec_heat}), (\ref{Main_Prop_Ass1}) and (\ref{dec_d2vaL2_ass}) we have
	\begin{align*}
	& \int_0^t\lV e^{-(t-s)L} \nab^4 u^{(c)}(s)\rV_{H^{N(|a|+1)}}ds\\
	\lesssim & \int_0^{t/2} (t-s)^{-2}\lV u^{(c)}\rV_{H^{N(|a|+1)}}ds+\int_{t/2}^t \<t-s\>^{-1} \lV\nab^2 u^{(c)}\rV_{H^{N(|a|+1)}}ds\\
	\lesssim & \ep_1 t^{-1}+ \ep_1 t^{-1/2+(3+|c|\delta)}\log t
	\lesssim  \ep_1 t^{-1/2+(3+|a|)\delta}.
	\end{align*}
	 For the last bound (\ref{dec_d2vL2_2}), the contributions of the error terms $Z^a(Err11+Err12)$ can be estimated as same as the quadratic terms, then it suffices to estimate the following two cases
	 \begin{equation*}
	 \int_0^{t/2}\lV  e^{-(t-s)L} \nab^2 \tilde{f}^{(a)}(s)\rV_{H^{N(|a|+1)}}ds+\int_{t/2}^t\lV  e^{-(t-s)L} \nab^2 \tilde{f}^{(a)}(s)\rV_{H^{N(|a|+1)}}ds:=I+II,
	 \end{equation*}
	 where $\tilde{f}^{(a)}$ denotes the quadratic terms in $f^{(a)}$. For $I$ we ultilize (\ref{Dec_heat}) to gain decay. Indeed, by (\ref{Dec_heat}) and (\ref{Main_Prop_Ass1}) we have
	 \begin{align*}
	 I\lesssim &\int_0^{t/2} (t-s)^{-3/2}\lV \phi^{(b)}\nab u^{(c)}+u^{(b)}u^{(c)}\rV_{H^{N(|a|+1)}}+(t-s)^{-9/4}\lV\Phi^{(b)}\Phi^{(c)}\rV_{W^{N(|a|+1),1}}ds\\
	 \lesssim & t^{-1/2}(\lV\Phi^{(b)}\rV_{H^{N(|a|)}}+\lV u^{(b)}\rV_{H^{N(|a|)}})\lV u^{(c)}\rV_{H^{N(|a|)}}+t^{-5/4}\lV\Phi^{(b)}\rV_{H^{N(a)}}\lV\Phi^{(c)}\rV_{H^{N(a)}}\\
	 \lesssim &\ep_1^2 t^{-1/2+2\delta}+\ep_1^2 t^{-5/4+4\delta}\lesssim \ep_1^2 t^{-1/2+2\delta}.
	 \end{align*}
	 For $II$ we ultilize (\ref{Dec_Phi}) to gain decay. By (\ref{Dec_heat}) we have
	 \begin{align*}
	 II\lesssim &\int_{t/2}^t \<t-s\>^{-3/2}\lV \phi^{(b)}\nab u^{(c)}\rV_{H^{N(|a|+1)+3}}+\<t-s\>^{-3/8}(t-s)^{-3/4}\lV u^{(b)}u^{(c)}\rV_{H^{N(|a|+1)+3/4}}\\
	 &+\<t-s\>^{-3/2}\lV\Phi^{(b)}\Phi^{(c)}\rV_{H^{N(|a|+1)+3}}ds\\
	 \lesssim & \sup_{s\in[t/2,t]} (\lV\Phi^{(b)}(s)\rV_{H^{N(|a|)}}\lV\nab u^{(c)}(s)\rV_{L^{\infty}}+\lV\phi^{(b)}(s)\rV_{L^{\infty}}\lV u^{(c)}(s)\rV_{H^{N(|a|)}})\\
	 &+\sup_{s\in[t/2,t]}\lV u^{(b)}(s)\rV_{H^2} \lV\nab^2 u^{(c)}(s)\rV_{H^{N(|a|+1)}}+ \ep_1^2 t^{-1+4\delta}.
	 \end{align*}
	 We then use (\ref{Main_Prop_Ass1}), (\ref{Dec_Phi}) and (\ref{phi^2-Hna}) to bound this by
	 \begin{align*}
	 II\lesssim  & \ep_1^2 t^{\kappa(b)-1/2+(3+|c|)\delta}+\ep_1\lV \nab^2 u^{(a)}\rV_{H^{N(|a|+1)}}+\ep_1^2 t^{-1/2+2\delta}\\
	 \lesssim & \ep_1^2 t^{-1/2+(3+|a|)\delta}+\ep_1\sup_{s\in [t/2,t]}\lV \nab^2 u^{(a)}(s)\rV_{H^{N(|a|+1)}}.
	 \end{align*}
	 This completes the proof of (\ref{nab3_v(b,c)_Dec}).

	\emph{Step 2: Estimate (\ref{Hntdtv(a)}).} From $v^{(a)}$-equation in (\ref{Main_Sys_VecFie}) and (\ref{nab3_v(b,c)_Dec}) we have
	\begin{equation}    \label{Hndtva_pf}
	\begin{aligned}
	\lV\d_t v^{(a)}\rV_{H^{N(|a|+1)}}\lesssim &\ep_1\<t\>^{-1/2+(3+|a|)\delta}+\sum_{|b|+|c|\leq a}\big(\lV \nab (\phi^{(b)}\nab u^{(c)})\rV_{H^{N(|a|+1)}}+\lV u^{(b)}\cdot\nab u^{(c)}\rV_{H^{N(|a|+1)}}\\
	&+\lV \d_j(\nab \phi^{(b)}\d_j\phi^{(c)})\rV_{H^{N(|a|+1)}}\big)+\lV Err11^{(a)}\rV_{H^{N(|a|+1)}}+\lV Err12^{(a)}\rV_{H^{N(|a|+1)}}.
	\end{aligned}
	\end{equation}
	The estimate of $\d_j(\nab \phi^{(b)}\d_j\phi^{(c)})$ is obtained by (\ref{Dec_Phi}) and (\ref{phi^2-Hna}), we then estimate other terms. By (\ref{nab3_v(b,c)_Dec}) it follows that
	\begin{align*}
	\lV \nab (\phi^{(b)}\nab u^{(c)})\rV_{H^{N(|a|+1)}}
	\lesssim & \lV\nab\phi^{(b)}\rV_{H^{N(|a|+1)}}\lV\nab^2 u^{(c)}\rV_{H^{N(|a|+1)}}\\
	\lesssim &\ep_1^2 t^{-1/2+\ka(b)\delta+(3+|c|)\delta}
	\lesssim  \ep_1^2 t^{-1/2+(3+|a|)\delta},
	\end{align*}
	and
	\begin{align*}
	\lV u^{(b)}\cdot\nab u^{(c)}\rV_{H^{N(|a|+1)}}\lesssim \lV u^{(b)}\rV_{H^{N(|a|+1)}}\lV\nab^2 u^{(c)}\rV_{H^{N(|a|+1)}}
	\lesssim \ep_1^2 t^{-1/2+(3+|a|)\delta}.
	\end{align*}
	The error terms $Err11$ and $Err12$ can be estimated similarly. Hence, the bound (\ref{Hntdtv(a)}) follows.
	
	Finally, from (\ref{xidxif}) and (\ref{Main_Prop_Ass1}), we have for any $|a|\leq N_1-1$
	\begin{align*}
	\lV\FF^{-1}(|\xi|\nab_{\xi}\widehat{u^{(a)}})\rV_{H^{N(|a|+1)}}\lesssim &\lV Su^{(a)}\rV_{H^{N(|a|+1)}}+\lV\Om u^{(a)}\rV_{H^{N(|a|+1)}}+\lV u^{(a)}\rV_{H^{N(|a|+1)}}+\lV t\d_t u^{(a)}\rV_{H^{N(|a|+1)}}\\
	\lesssim & \ep_1 + \lV t\d_t v^{(a)}\rV_{H^{N(|a|+1)}}.
	\end{align*}
	 Then by (\ref{Hntdtv(a)}) we obtain the bound (\ref{nab_xi_v(b,c)_HN}). This completes the proof of the Lemma.
	
\end{proof}

Next, we prove the decay estimates of $v$.
\begin{lemma}       \label{Dec_lem}
	With the notations and hypothesis in Proposition \ref{Main_Prop}, for any $t\in[0,T]$, $|a|\leq N_1$ and $p=\frac{3}{2\delta}$, we have
	\begin{gather}            \label{Dec_v_VecFie}
	\lV v^{(a)}\rV_{W^{N(|a|+1),p}}\lesssim \ep_1\<t\>^{-3/4+(1+|a|)\delta},
	\end{gather}
\end{lemma}
\begin{proof}
	Here it suffices to prove (\ref{Dec_v_VecFie}) for $t> 1$, otherwise this estimate is obtained by Sobolev embedding and (\ref{Main_Prop_Ass1}). We prove (\ref{Dec_v_VecFie}) by induction. Assume that
\begin{equation}       \label{Decva_ass}
\lV v^{(c)}\rV_{W^{N(|c|+1),p}}\lesssim \ep_1\<t\>^{-3/4+(1+|c|)\delta},\ \mathrm{for}\  \ |c|<|a|.
\end{equation}
By (\ref{Dec_heat}) and Duhamel's formula, it suffices to prove
\begin{gather}
\int_0^t \lV e^{-(t-s)L} \Delta u^{(c)}\rV_{W^{N(|a|+1),p}}ds\lesssim \ep_1^2 t^{-3/4+(1+|a|)\delta},\ \mathrm{for}\ |c|<|a|,\label{Dec_Pf_Delva}\\
\int_0^t\lV e^{-(t-s)L} \tilde{f}^{(a)}(s)\rV_{W^{N(|a|+1),p}}ds\lesssim \ep_1\sup_{s\in[t/2,t]}\lV v^{(a)}(s)\rV_{W^{1,p}} +\ep_1^2 t^{-3/4+(1+|a|)\delta},\label{Dec_Pf_phiua}.  
\end{gather}
where $\tilde{f}^{(a)}$ denotes the quadratic terms, i.e. 
\begin{equation*}
\tilde{f}^{(a)}=\nab(\phi^{(b)}\nab u^{(c)})+u^{(b)}\cdot\nab u^{(c)}+\nab(\Phi^{(b)}\Phi^{(c)}),\ |b|+|c|\leq |a|.
\end{equation*}

For the first bound (\ref{Dec_Pf_Delva}), it follows from (\ref{Dec_heat}), (\ref{Main_Prop_Ass1}) and (\ref{Decva_ass}) that
\begin{align*}
& \int_0^t\lV e^{-(t-s)H} \Delta u^{(c)}(s)\rV_{W^{N(|a|+1),p}}ds\\
\lesssim & \int_0^{t/2} (t-s)^{-7/4+3/2p} \lV u^{(c)}\rV_{H^{N(|a|+1)}} ds
+\int_{t/2}^{t} \<t-s\>^{-1} \lV  u^{(c)}\rV_{W^{N(|a|+1)+2,p}} ds\\
\lesssim &\ep_1 t^{-3/4+3/2p}+\log t \sup_{s\in[t/2,t]}\lV v^{(c)}(s)\rV_{W^{N(|a|+1)+2,p}}\\
\lesssim & \ep_1 t^{-3/4+\delta}+\ep_1 t^{-3/4+(2+|c|)\delta}\lesssim  \ep_1 t^{-3/4+(1+|a|)\delta}.
\end{align*}
To prove the second bound (\ref{Dec_Pf_phiua}), we divide the left-hand side of (\ref{Dec_Pf_phiua}) into
\begin{equation*}
{\rm LHS (\ref{Dec_Pf_phiua})}=\int_0^{t/2}\lV e^{-(t-s)L} \tilde{f}^{(a)}(s)\rV_{W^{N(|a|+1),p}}ds+\int_{t/2}^t\lV e^{-(t-s)L} \tilde{f}^{(a)}(s)\rV_{W^{N(|a|+1),p}}ds:=I+II.
\end{equation*}
We ultilize (\ref{Dec_heat}) and (\ref{Main_Prop_Ass1}) to bound $I$ by
\begin{align*}
I\lesssim & \int_0^{t/2} (t-s)^{-5/4+3/2p}\lV\phi^{(b)}\nab u^{(c)}+ u^{(b)}\cdot  u^{(c)}\rV_{H^{N(|a|+1)}}+(t-s)^{-2+3/2p} \lV\nab\phi^{(b)}\nab\phi^{(c)})\rV_{W^{N(|a|+1),1}}ds\\
\lesssim & t^{-5/4+3/2p}\int_0^{t/2}(\lV\Phi^{(b)}\rV_{H^{N(|a|)}}+\lV u^{(b)}\rV_{H^{N(|a|)}})\lV\nab u^{(c)}\rV_{H^{N(|a|)}}ds+t^{-1+\delta} \lV\Phi^{(b)}\rV_{H^{N(a)}}\lV\Phi^{(c)}\rV_{H^{N(a)}}\\
\lesssim & \ep_1^2 t^{-3/4+(1+|a|)\delta}+\ep_1^2 t^{-1+4\delta}\lesssim \ep_1^2 t^{-3/4+(1+|a|)\delta}.
\end{align*}
Next for the other term $II$, we consider the contributions of $\nab(\phi^{(b)}\nab u^{(c)})$, $u^{(b)}\cdot\nab u^{(c)}$ and $\nab(\Phi^{(b)}\Phi^{(c)})$, respectively. Namely, 
\begin{align*}
II=&\int_{t/2}^t\lV e^{-(t-s)L} \nab(\phi^{(b)}\nab u^{(c)})\rV_{W^{N(|a|+1),p}}ds+\int_{t/2}^t\lV e^{-(t-s)L} u^{(b)}\cdot\nab u^{(c)}\rV_{W^{N(|a|+1),p}}ds\\
&+\int_{t/2}^t\nab(\Phi^{(b)}\Phi^{(c)})\rV_{W^{N(|a|+1),p}}ds:=II_1+II_2+II_3.
\end{align*}

First, we estimate $II_1$. If $c=a$, using (\ref{Dec_heat}), (\ref{Hardy_2}), (\ref{Main_Prop_Ass1}), (\ref{Dec_Phi/nab}) and (\ref{nab3_v(b,c)_Dec}) it follows that
\begin{align*}
II_{11}:=&\int_{t/2}^t \lV e^{-(t-s)L} \nab(\phi^{(b)}\nab u^{(c)})\rV_{\dot{W}^{N(|a|+1),p}} ds\\
\lesssim & \int_{t/2}^{t-1} (t-s)^{-3} \lV \phi\nab u^{(a)}\rV_{\dot{W}^{N(|a|+1)-5,p}} ds
+\int_{t-1}^t (t-s)^{-3/4} \lV \phi\nab u^{(a)}\rV_{\dot{W}^{N(|a|+1)-1/2,p}} ds\\
\lesssim & \sup_{s\in[t/2,t]} (\lV\Phi(s)\rV_{H^{N(a)}}\lV\nab u^{(a)}(s)\rV_{L^p}+\lV\phi(s)\rV_{\Lf} \lV \nab^2 u^{(a)}(s)\rV_{H^{N(|a|+1)}})\\
\lesssim & \ep_1\sup_{s\in[t/2,t]}\lV\nab v^{(a)}(s)\rV_{L^p} +\ep_1^2 t^{-1+(5+|a|)\delta}.
\end{align*}
And by integration by parts, (\ref{Dec_heat}), (\ref{Main_Prop_Ass1}), (\ref{Dec_Phi/nab}) and (\ref{nab3_v(b,c)_Dec}), we get
\begin{align*}
II_{12}:=&\int_{t/2}^t \lV e^{-(t-s)L} P_{<0} \nab(\phi^{(b)}\nab u^{(c)})\rV_{L^p} ds\\
\lesssim &\int_{t/2}^t \lV e^{-(t-s)L} P_{<0}[\nab^2(\phi u^{(a)})-\nab(\nab\phi u^{(a)})]\rV_{L^p} ds\\
\lesssim & \int_{t/2}^t \<t-s\>^{-1} \lV\phi u^{(a)}\rV_{W^{2,p}}ds +\int_{t/2}^t \<t-s\>^{-5/4+3/2p} \lV\nab\phi u^{(a)}\rV_{H^1} ds\\
\lesssim & \ep_1 t^{-1/2+2\delta}\log t \sup_{s\in[t/2,t]}\lV u^{(a)}(s)\rV_{W^{2,p}}+\ep_1^2 t^{-1+2\delta}\\
\lesssim & \ep_1 \sup_{s\in[t/2,t]} \lV v^{(a)}(s)\rV_{W^{2,p}}+\ep_1^2 t^{-1+2\delta}.
\end{align*}
If $|c|<|a|$, using (\ref{Dec_heat}), (\ref{Dec_Phi}), (\ref{Main_Prop_Ass1}) and (\ref{Decva_ass}), we have
\begin{align*}
II_{11}\lesssim & \int_{t/2}^{t} \<t-s\>^{-3} \lV \phi^{(b)}\nab u^{(c)}\rV_{W^{N(|a|+1)+1,p}} ds
\lesssim  \sup_{s\in[t/2,t]} \lV\Phi^{(b)}(s)\rV_{H^{N(a)}}\lV v^{(c)}(s)\rV_{W^{N(a),p}}\\
\lesssim & \ep_1^2 t^{\ka(b)\delta-3/4+(1+|c|)\delta}\lesssim \ep_1^2 t^{-3/4+(1+|a|)\delta},
\end{align*}
and
\begin{align*}
II_{12}\lesssim &\int_{t/2}^t \lV e^{-(t-s)L}P_{<0} [\nab^2(\phi^{(b)} u^{(c)})-\nab(\nab\phi^{(b)} u^{(c)})]\rV_{L^p} ds\\
\lesssim & \int_{t/2}^t \<t-s\>^{-1-3/2p} \lV\phi^{(b)} u^{(c)}\rV_{W^{2,p/2}}ds +\int_{t/2}^t \<t-s\>^{-5/4+3/2p} \lV\nab\phi^{(b)} u^{(c)}\rV_{H^1} ds\\
\lesssim & \ep_1^2 t^{\ka(b)\delta-3/4+(1+|c|)\delta}\lesssim \ep_1^2 t^{-3/4+(1+|a|)\delta}.
\end{align*}
These are acceptable for $II_1$. 

In order to estimate $II_2$, by $\text{div} u^{(b)}=0$ and (\ref{Dec_heat}) we have
\begin{align*}
II_2\lesssim  \int_{t/2}^t \<t-s\>^{-\frac{5}{4}+\frac{3}{2p}} \lV  u^{(b)}\cdot  u^{(c)}\rV_{H^{N(|a|+1)+3}} ds
\lesssim  \sup_{s\in[t/2,t]}\lV u^{(b)}u^{(c)}\rV_{H^{N(|a|+1)+3}}.
\end{align*}
If $|a|=0$, by (\ref{Main_Prop_Ass1}) we have
\begin{equation*}
\sup_{s\in[t/2,t]}\lV u^2\rV_{H^{N(|a|+1)+3}}\lesssim \sup_{s\in[t/2,t]}\lV u\rV_{H^{N(0)}}\lV u\rV_{L^{\infty}}\lesssim \ep_1 \sup_{s\in[t/2,t]}\lV v\rV_{W^{1,p}}.
\end{equation*}
If $|a|>0$, by symmetry we may assume that $|c|< |a|$, then from (\ref{Main_Prop_Ass1}) and (\ref{Decva_ass}) we obtain 
\begin{equation*}
\sup_{s\in[t/2,t]}\lV u^{(b)}u^{(c)}\rV_{H^{N(|a|+1)+3}}\lesssim \sup_{s\in[t/2,t]}\lV u^{(b)}\rV_{H^{N(a)}}\lV  u^{(c)}\rV_{W^{N(|c|+1),p}}\lesssim \ep_1^2 t^{-3/4+(1+|c|)\delta}.
\end{equation*}
Thus the contribution of $II_2$ is acceptable. 

Finally, for the term $II_3$, by (\ref{Dec_heat}), (\ref{Main_Prop_Ass1}) and (\ref{phi^2-Hna}) we obtain
\begin{align*}
II_3\lesssim & \int_{t/2}^t \<t-s\>^{-5/4+3/2p} \lV\Phi^{(b)}\Phi^{(c)}\rV_{H^{N(|a|)}}ds
\lesssim \ep_1^2\<t\>^{-1+4\delta}.
\end{align*}
This concludes the bound (\ref{Dec_Pf_phiua}), and hence completes the proof of the lemma.	
\end{proof}

Base on the linear decay estimates (\ref{Dec_heat}), we prove the following decay estimates of $\nab^l v$ for $l= 1,2$.
\begin{lemma}  \label{Dec_lem_VF}
	With the notations and hypothesis in Proposition \ref{Main_Prop}. For any $t\in [0,T]$, $l=1,2$ and $p=\frac{3}{2\delta}$. If $|a|\leq N_1-2$, we have
	\begin{equation}        \label{Dec_d2v_Bet}
	\lV \nab^l v^{(a)}\rV_{W^{N(|a|+2)-l,p}}\lesssim \ep_1\<t\>^{-5/4+(4+|a|)\delta}.
	\end{equation}
	If $|a|= N_1-1\ or\ N_1$, we have
	\begin{gather}        \label{Dec_nab2v_VF}
	\lV \nab^l v^{(a)}\rV_{W^{N(|a|+2)-l,p}}\lesssim \ep_1\<t\>^{-1+(6+|a|)\delta}.
	\end{gather}
\end{lemma}
\begin{proof}
	\emph{Step 1: Proof of (\ref{Dec_d2v_Bet}).} We prove the bound (\ref{Dec_d2v_Bet}) by induction. Assume that
	\begin{equation}    \label{dlv_ass}
	\lV \nab^l v^{(c)} \rV_{W^{N(|a|+2)-l,p}} \lesssim \ep_1\<t\>^{-5/4+(4+|c|)\delta},\ \ \mathrm{for}\ |c|<|a|.
	\end{equation}
	Then by (\ref{Dec_heat}) and Duhamel's formula, it suffices to prove that
	\begin{gather}
	\lV \int_0^t e^{-(t-s)L} \nab^{l+2} u^{(c)}\rV_{W^{N(|a|+2)-l,p}}\lesssim \ep_1^2 t^{-5/4+(4+|a|)\delta},\ \mathrm{for}\ |c|<|a|,\label{dec_dlv_2}\\
	\label{dec_dlv_3}
	\lV \int_0^t e^{-(t-s)L} \nab^{l+1} (\phi^{(b)}\nab u^{(c)})\rV_{W^{N(|a|+2)-l,p}}\lesssim \ep_1^2 \<t\>^{-5/4+(4+|c|)\delta},\\
	\lV \int_0^t e^{-(t-s)L} \nab^l ( u^{(b)}\cdot \nab u^{(c)})\rV_{W^{N(|a|+2)-l,p}}\lesssim \ep_1^2 \<t\>^{-5/4+\delta}, \label{dec_dlv_4}
	\end{gather}
	and
	\begin{gather}\label{dec_dlv_5}
	\lV \int_0^t e^{-(t-s)L} \nab^{l+1}(\nab\phi^{(b)}\nab\phi^{(c)})ds\rV_{W^{N(|a|+2)-l,p}}\lesssim \ep_1^2 t^{-5/4},\ \mathrm{for}\ |a|\leq N_1-2.
	\end{gather}
	
	Next, we prove the above four bounds respectively. From (\ref{Dec_heat}), (\ref{Main_Prop_Ass1}) and the assumption (\ref{dlv_ass}) we have
	\begin{equation*}
	\begin{aligned}
	{\rm LHS}(\ref{dec_dlv_2})\lesssim & \int_0^{t/2}(t-s)^{-3/4+3/2p-l/2-1} \lV u^{(c)}\rV_{H^{N(|a|+2)}} ds+\int_{t/2}^t \<t-s\>^{-1}\lV\nab^l u^{(c)}\rV_{W^{N(|a|+2)+1,p}} ds\\
	\lesssim &\ep_1 t^{-3/4-l/2+3/2p}+\ep_1 t^{-5/4+(4+|c|)\delta}\log t
	\lesssim  \ep_1 t^{-5/4+(4+|a|)\delta}.
	\end{aligned}
	\end{equation*}
	This implies the bound (\ref{dec_dlv_2}).
	
	For the bound (\ref{dec_dlv_3}) with $|b|+|c|\leq N_1-2$. Using (\ref{Dec_heat}) and (\ref{Main_Prop_Ass1}), we have
	\begin{equation}        \label{phi.u-5/4}
	\begin{aligned}
	& \int_0^{t/2}\lV e^{-(t-s)L} \nab^{l+1} (\phi^{(b)}\nab u^{(c)})\rV_{W^{N(|a|+2)-l,p}}ds\\
	\lesssim &\int_0^{t/2} (t-s)^{-5/4+3/2p-l/2} \lV \phi^{(b)}\nab u^{(c)}\rV_{H^{N(|a|+2)}}ds\\
	\lesssim & t^{-5/4+3/2p-l/2} \int_0^{t/2}  \lV\phi^{(b)}\rV_{L^{\infty}}\lV\nab u^{(c)}\rV_{H^{N(|a|+2)}}+\lV\phi^{(b)}\rV_{\dot{H}^{N(|a|+2)}}\lV\nab u^{(c)}\rV_{L^{\infty}}ds\\
	\lesssim & t^{-5/4+3/2p-l/2} \int_0^{t/2} \lV\Phi^{(b)}\rV_{H^{N(|a|+2)}}\lV\nab u^{(c)}\rV_{H^{N(|a|+2)}}ds
	\lesssim  t^{-5/4+3/2p+2\delta}.
	\end{aligned}
	\end{equation}
	By (\ref{Dec_heat}), (\ref{Hardy_2}), (\ref{Dec_Phi}) and (\ref{Dec_v_VecFie}), we have
	\begin{align*}        
	&\int_{t/2}^t\lV  e^{-(t-s)L} \nab^{l+1} (\phi^{(b)}\nab u^{(c)})\rV_{W^{N(|a|+2)-l,p}}ds\\
	\lesssim & \int_{t/2}^t \<t-s\>^{\frac{l+1}{2}}\lV \phi^{(b)}\nab u^{(c)}\rV_{W^{N(|a|+2)+1,p}}ds\\
	\lesssim & \log t \sup_{s\in[t/2,t]}( \lV \phi^{(b)}\rV_{\Lf}\lV u^{(c)}\rV_{W^{N(|a|+1),p}}+\lV \nab\phi^{(b)}\rV_{\dot{W}^{N(|a|+2),p}}\lV\nab u^{(c)}\rV_{\Lf})\\
	\lesssim & t^{\delta}(\ep_1^2\<t\>^{-1/2+2\delta-3/4+(1+|c|)\delta}+\ep_1^2\<t\>^{-1+4\delta-3/4+(1+|c|)\delta})\lesssim \ep_1^2\<t\>^{-5/4+(3+|c|)\delta}.
	\end{align*}
	Hence, the bound (\ref{dec_dlv_3}) follows. 
	
	Next, for the bound (\ref{dec_dlv_4}) with $|b|+|c|\leq N_1$. using (\ref{Dec_heat}), (\ref{Main_Prop_Ass1}) and (\ref{Dec_v_VecFie}), we have
	\begin{align*}
	&\int_0^t\lV  e^{-(t-s)L}  \nab^l(u^{(b)}\cdot\nab u^{(c)})\rV_{W^{N(|a|+2)-l,p}}ds\\
	\lesssim &\int_0^{t/2} (t-s)^{-3/4+3/2p-l/2}\lV  u^{(b)}\cdot\nab u^{(c)}\rV_{H^{N(|a|+2)}} ds+\int_{t/2}^t \<t-s\>^{\frac{l+1}{2}}\lV u^{(b)}u^{(c)}\rV_{W^{N(|a|+2)+1,p}}ds\\
	\lesssim & \<t\>^{-5/4+\delta}\int_0^{t/2}\lV\nab u^{(b)}\rV_{H^{N(a)}}\lV\nab u^{(c)}\rV_{H^{N(a)}}ds+\log t \sup_{s\in[t/2,t]} \lV  u^{(b)}(s)\rV_{W^{N(|a|+1),p}} \lV u^{(c)}(s)\rV_{W^{N(|a|+1),p}}\\
	\lesssim &  \ep_1^2t^{-5/4+\delta}+\ep_1^2\<t\>^{-3/2+(3+|a|)\delta}
	\lesssim \ep_1^2 t^{-5/4+\delta},
	\end{align*}

	Finally, for the bound (\ref{dec_dlv_5}) with $|b|+|c|\leq N_1-2$. By (\ref{Dec_heat}), (\ref{Main_Prop_Ass1}) and (\ref{Dec_Phi}) we have
	\begin{equation*}
	\begin{aligned}
	&\int_0^t\lV  e^{-(t-s)L} \nab^{l+1} (\nab\phi^{(b)}\nab\phi^{(c)})\rV_{W^{N(|a|+2)-l,p}}ds\\
	\lesssim &\int_0^{t/2} (t-s)^{-\frac{3}{2}+\frac{3}{2p}-\frac{l+1}{2}}\lV \Phi^{(b)}\Phi^{(c)}\rV_{W^{N(|a|+2),1}}ds+\int_{t/2}^{t-1} \<t-s\>^{-(l+1)/2}\lV \Phi^{(b)}\Phi^{(c)}\rV_{W^{N(|a|+2)-l,p}} ds\\
	&+\int_{t-1}^t (t-s)^{-1/2}\lV \Phi^{(b)}\Phi^{(c)}\rV_{W^{N(|a|+2),p}} ds\\
	\lesssim & \ep_1^2 t^{-5/4}+ \ep_1^2t^{-2+8\delta}\log t\lesssim \ep_1^2 t^{-5/4}.
	\end{aligned}
	\end{equation*}
	Hence, the bound (\ref{dec_dlv_5}) follows. This completes the proof of (\ref{Dec_d2v_Bet}).

	\emph{Step 2: Proof of (\ref{Dec_nab2v_VF}).} We prove (\ref{Dec_nab2v_VF}) by induction. From (\ref{Dec_d2v_Bet}) we may assume that
	\begin{equation}    \label{d2v_ass}
	\lV \nab^l v^{(c)} \rV_{W^{N(|c|+2)-l.p}} \lesssim \ep_1\<t\>^{-1+(6+|c|)\delta},\ \ \mathrm{for}\ |c|<|a|.
	\end{equation}
	By the assumption and (\ref{Dec_heat}), we have the bounds (\ref{dec_dlv_4}) and 
	\begin{equation*}
	\lV \int_0^t e^{-(t-s)L} \nab^{l+2} u^{(c)}\rV_{W^{N(|a|+2)-l,p}}\lesssim \ep_1^2 t^{-1+(6+|a|)\delta},\ \mathrm{for}\ |c|<|a|\leq N_1.
	\end{equation*}
	Then by Duhamel's formula, it suffices to prove that for any $|b|+|c|\leq |a|$
	\begin{gather}         	\label{J32}
	\int_0^t\lV e^{-(t-s)L} \nab^{l+1} (\phi^{(b)}\nab u^{(c)})\rV_{W^{N(|a|+2)-l,p}}ds\lesssim \ep_1^2 \<t\>^{-1+(6+|a|)\delta},\\ \label{J33}
	\lV \int_0^t e^{-(t-s)L} \nab^{l+1}(\nab\phi^{(b)}\nab\phi^{(c)})ds\rV_{W^{N(|a|+2)-l,p}}\lesssim 	\ep_1^2 t^{-1+4\delta},\ \mathrm{for}\ N_1-1\leq|a|\leq N_1,
	\end{gather}
	
	For the first bound (\ref{J32}), if $|c|<|a|$, using (\ref{Dec_heat}) and (\ref{d2v_ass}) we have
	\begin{align*}
	&\int_{t/2}^t\lV  e^{-(t-s)L} \nab^{l+1} (\phi^{(b)}\nab u^{(c)})\rV_{W^{N(|a|+2)-l,p}}ds\\
	\lesssim &\int_{t/2}^t \<t-s\>^{-(l+1)/2-3/2p} \lV\phi^{(b)}\nab u^{(c)}\rV_{W^{N(|a|+2)+4,p/2}}ds\\
	\lesssim  &\sup_{s\in[t/2,t]}\lV\Phi^{(b)}(s)\rV_{H^{N(|a|+1)}}\lV\nab^l u^{(c)}(s)\rV_{W^{N(|a|+1),p}}\lesssim \ep_1^2 t^{\ka(b)\delta-1+(6+|c|)\delta}\lesssim \ep_1^2 t^{-1+(6+|a|)\delta};
	\end{align*}
	If $c=a$, we have to use the bounds (\ref{Hardy_1}), (\ref{nab3_v(b,c)_Dec}) and (\ref{Dec_Phi/nab}),
	\begin{align*}
	&\int_{t/2}^t\lV  e^{-(t-s)L} \nab^{l+1} (\phi^{(b)}\nab u^{(c)})\rV_{W^{N(|a|+2)-l,p}}ds\\
	\lesssim &\int_{t/2}^t \<t-s\>^{-(l+1)/2}\lV\phi\nab u^{(a)}\rV_{W^{N(|a|+2)+1,p}} ds\\
	\lesssim &\log t \sup_{s\in[t/2,t]} (\sum_{k\geq 0} 2^{N(|a|+2)k}\lV P_k\Phi(s)\rV_{p}+\lV\phi(s)\rV_{\infty})\lV \nab^2 u^{(a)}(s)\rV_{H^{N(|a|+1)}}\\
	\lesssim &t^{\delta}(\ep_1t^{-1+4\delta}+\ep_1t^{-1/2+2\delta})\ep_1t^{-1/2+(3+|a|)\delta}\lesssim \ep_1^2 t^{-1+(6+|a|)\delta}.
	\end{align*}
	which, together with (\ref{phi.u-5/4}), gives the bound (\ref{J32}). 
	
	For the second bound (\ref{J33}) with $|a|=N_1-1\ or\ N_1$, from (\ref{Dec_heat}), (\ref{Hardy_1}), (\ref{Main_Prop_Ass1}) and (\ref{phi^2-Hna}) we have
	\begin{align*}
	&\int_0^t\lV  e^{-(t-s)L} \nab^l \d_j(\nab\phi^{(b)}\d_j\phi^{(c)})\rV_{W^{N(|a|+2)-l,p}}ds\\
	\lesssim &\int_0^{t/2} (t-s)^{-\frac{3}{2}+\frac{3}{2p}-\frac{l+1}{2}}\lV \Phi^{(b)}\Phi^{(c)}\rV_{W^{N(|a|+2),1}}ds
	+\int_{t/2}^t \<t-s\>^{-3/4+3/2p-(l+1)/2}\lV \Phi^{(b)}\Phi^{(c)}\rV_{H^{N(|a|+1)}} ds\\
	\lesssim & \ep_1^2 t^{-1}+ \ep_1^2 t^{-1+\ka(|b|+2)\delta+\delta+\ka(c)\delta}\lesssim \ep_1^2 t^{-1+4\delta}.
	\end{align*}
	This concludes the bound (\ref{Dec_nab2v_VF}).

\end{proof}

As a consequence of (\ref{Dec_d2v_Bet}) and (\ref{Dec_nab2v_VF}), we obtain the following estimates.
\begin{lemma}  
	With the notations and hypothesis in Proposition \ref{Main_Prop}. For any $t\in [0,T]$ and $p=\frac{3}{2\delta}$, we have
	\begin{gather}      \label{Dec_dtv_VF-bett}
	\lV \d_t v^{(a)}\rV_{W^{N(|a|+2)-2,p}}\lesssim\ep_1\<t\>^{-5/4+(4+|a|)\delta},\ \ \mathrm{if}\ |a|\leq N_1-2,\\          \label{Dec_dtv_VF}
	\lV \d_t v^{(a)}\rV_{W^{N(|a|+2)-2,p}}\lesssim\ep_1\<t\>^{-1+(6+|a|)\delta},\ \ \mathrm{if}\ N_1-1\leq|a|\leq N_1.
	\end{gather}
\end{lemma}

Finally, we prove the following $H^n$-norm estimates.
\begin{lemma}         \label{dt_v_2norm_Lem}
	With the notations and hypothesis in Proposition \ref{Main_Prop}. For any $t\in[0,T]$, $|a|\leq N_1$, we have
	\begin{gather} \label{H^n_dtv}
	\lV\d_t v^{(a)}\rV_{H^n}\lesssim \<t\>^{2\delta}\sum_{|c|\leq |a|}\lV \nab u^{(c)}\rV_{H^{n+1}}+\ep_1^2\<t\>^{-1+5\delta},\ \ \mathrm{for}\ n\leq N(|a|)-1,\\   \label{H^n-1_dttphi^a}
	\lV\d_{t} \Psi^{(a)}\rV_{H^n}+\lV\nab^2 \phi^{(a)}\rV_{H^n}\lesssim \ep_1\<t\>^{4\delta}(1+\sum_{|c|\leq |a|}\lV \nab u^{(c)}\rV_{H^{n+1}}), \ \mathrm{for}\ n=N(|a|)-1,N(|a|),\\
	 \label{dt(Psi)_2norm_dec}
	\lV\d_t \Psi^{(a)}\rV_{H^{N(|a|)-3}}\lesssim \ep_1^2 \<t\>^{-3/4+(|a|+1)\delta}.
	\end{gather}
\end{lemma}
\begin{proof}First, we prove (\ref{H^n_dtv}). In fact, we have from $u^{(a)}$-equation in (\ref{Main_Sys_VecFie}) that 	\begin{align*}
	\lV\d_t v^{(a)}\rV_{H^n}\lesssim &\sum_{|c|\leq |a|}\lV \nab u^{(c)}\rV_{H^{n+1}}+\sum_{|b|+|c|\leq |a|}\big[\lV u^{(b)}\cdot\nab u^{(c)}\rV_{H^n}+\lV \d_j(\phi^{(b)}\nab u^{(c)})\rV_{H^n}\\
	&+\lV\d_j(\nab\phi^{(b)}\d_j\phi^{(c)})\rV_{H^n}\big]+\sum_{|c|\leq |a|}\lV \P Z^c (Err11+Err12)\rV_{H^n}.
	\end{align*}
	Moreover, from (\ref{Main_Prop_Ass1}) it follows that
	\begin{equation}
	\lV u^{(b)}\cdot\nab u^{(c)}\rV_{H^n}\lesssim \lV u^{(b)}\rV_{H^n}\lV\nab u^{(c)}\rV_{H^n}\lesssim \ep_1 \lV\nab u^{(c)}\rV_{H^n},
	\end{equation}
	and
	\begin{align*}
	\lV \nab(\phi^{(b)}\nab u^{(c)})\rV_{H^n}\lesssim & \lV \nab(\phi^{(b)}\nab u^{(c)})\rV_{L^2}+\lV \d_j(\phi^{(b)}\nab u^{(c)})\rV_{\dot{H}^n}\\
	\lesssim & \lV\nab\phi^{(b)}\rV_{H^2}\lV\nab u^{(c)}\rV_{H^1}+\lV\phi^{(b)}\rV_{\dot{H}^{n+1}}\lV \nab u^{(c)}\rV_{L^{\infty}}+\lV\phi^{(b)}\rV_{L^{\infty}}\lV\nab u^{(c)}\rV_{\dot{H}^{n+1}}\\
	\lesssim & \lV \Phi^{(b)}\rV_{H^n}\lV \nab u^{(c)}\rV_{H^{n+1}}\\
	\lesssim & \ep_1 t^{2\delta}\lV \nab u^{(c)}\rV_{H^{n+1}}.
	\end{align*}
	The term $\lV Z^a Err11\rV_{H^n}$ can be estimated similarly. Finally, $\lV\P\d_j(\nab\phi^{(b)}\d_j\phi^{(c)})\rV_{H^n}$ has beed estimated in (\ref{phi^2-Hna}), and the term $\lV Z^a Err12\rV_{H^n}$ can be estimated using similar argument. Hence, the bound (\ref{H^n_dtv}) follows.
	
	Second, we prove (\ref{H^n-1_dttphi^a}) and (\ref{dt(Psi)_2norm_dec}). The bound (\ref{H^n-1_dttphi^a}) can be obtained directly from (\ref{H^n_dtv}), (\ref{Main_Sys_VecFie}) and (\ref{Main_Prop_Ass1}). Then we prove the bound (\ref{dt(Psi)_2norm_dec}). In view of (\ref{Main_Sys_VecFie}), it suffices to prove that
	\begin{equation}       \label{dtPsi_Dec}
	\begin{aligned}
	&\sum_{b+c=a}\big(\lV \d_t u^{(b)}\cdot\nab \phi^{(c)}\rV_{H^{N(a)-3}}+\lV  u^{(b)}\cdot\nab \d_t\phi^{(c)}\rV_{H^{N(a)-3}}\big)\\
	&+\sum_{b+c+e=a}\lV u^{(b)}\cdot\nab(u^{(c)}\cdot\nab \phi^{(e)})\rV_{H^{N(a)-3}}+\lV Z^a Err2\rV_{H^{N(a)-3}}\lesssim \ep_1^2\<t\>^{-3/4+(|a|+1)\delta}.
	\end{aligned}
	\end{equation}
	Here we only estimate the first term $\d_t u^{(b)}\cdot\nab \phi^{(c)}$ in detail, the other terms in (\ref{dtPsi_Dec}) are similar.
	
	\emph{Case 1: $|a|\geq N_1-1$.}
	
	If $|b|\leq |c|$, using (\ref{Dec_dtv_VF-bett}) and (\ref{Main_Prop_Ass1}), it follows that
	\begin{equation}\label{dtPsi_N1-1_b<c}
	\begin{aligned}
	\lV \d_t u^{(b)}\cdot\nab \phi^{(c)}\rV_{H^{N(a)-3}}\lesssim &(\sum_{k\geq 0}2^{(N(a)-3)k^+}\lV \d_t P_k u^{(b)}\rV_{L^{\infty}}+\lV \d_t u^{(b)}\rV_{L^{\infty}})\lV\Phi^{(c)}\rV_{H^{N(a)-3}}\\
	\lesssim &\ep_1\<t\>^{\ka(c)\delta}\lV \d_t u^{(b)}\rV_{W^{N(|b|+2)-2,p}}\\
	\lesssim & \ep_1\<t\>^{2\delta}\lV \d_t v^{(b)}\rV_{W^{N(|b|+2)-2,p}}\\
	\lesssim &\ep_1^2 \<t\>^{-1}.
	\end{aligned}
	\end{equation}
	If $|b|>|c|$, it follows from (\ref{Dec_Phi}) and (\ref{H^n_dtv}) that
	\begin{equation}     \label{dtPsi_N1-1_b>c}
	\lV \d_t u^{(b)}\cdot\nab \phi^{(c)}\rV_{H^{N(a)-3}}\lesssim \lV \d_t u^{(b)}\rV_{H^{N(a)-3}}\sum_k 2^{(N(a)-3)k^+}\lV\Phi^{(c)}\rV_{L^{\infty}}\lesssim \ep_1^2\<t\>^{-1+5\delta}.
	\end{equation}
	
	\emph{Case 2: $|a|\leq N_1-2$.}
	
	Using (\ref{Dec_dtv_VF}), (\ref{Dec_Phi}), (\ref{H^n_dtv}) and (\ref{Main_Prop_Ass1}), we get
	\begin{equation}      \label{dtPsi_N1-2}
	\lV \d_t u^{(b)}\cdot\nab \phi^{(c)}\rV_{H^{N(a)-3}}\lesssim \lV \d_t u^{(b)}\rV_{H^{N(a)-3}}\lV \Phi^{(c)}\rV_{L^{\infty}}+\lV \d_t u^{(b)}\rV_{L^{\infty}}\lV\Phi^{(c)}\rV_{H^{N(a)-3}}\lesssim \ep_1^2\<t\>^{-1+5\delta}.
	\end{equation}
	
	From the above two cases, the bound (\ref{dtPsi_Dec}) for the first term follows. This completes the proof of the Lemma.

\end{proof}

\section{Energy estimates}
In this section we prove the energy bounds (\ref{Main_Prop_result1}).
\subsection{The bound on $v$ and $\Phi$} We start with the Sobolev bound in (\ref{Main_Prop_result1}).

\begin{proposition}            \label{Prop_Ene_Sob}
	With the notation and hypothesis in Proposition $\ref{Main_Prop}$, for any $t\in[0,T]$, we have
	\begin{equation*}
	\lV v\rV_{H^{N(0)}}^2+\int_0^t \lV\nab v\rV_{H^{N(0)}}^2 ds+\lV\Phi\rV_{H^{N(0)}}^2\lesssim \epsilon_0^2.
	\end{equation*}
\end{proposition}
\begin{proof}
From (\ref{mu1mu2}) and (\ref{L-def}), the operator $L^{1/2}$ can be defined as
\begin{equation*}
L^{1/2}:=\text{diag} \Big\{\Big[\frac{\nu_4+\nu_5}{2}|\nab|^2+\nu_1 \frac{\d_1^2(\d_2^2+\d_3^2)}{|\nab|^2}\Big]^{\frac{1}{2}},\Big[\frac{\nu_4+\nu_5}{2}|\nab|^2+\nu_1 \frac{\d_1^2(\d_2^2+\d_3^2)}{|\nab|^2}\Big]^{\frac{1}{2}},\Big(\frac{\nu_4}{2}|\nab|^2-\nu_5 \frac{\d_1^2}{2}\Big)^{\frac{1}{2}}\Big\}.
\end{equation*}
Then we define the energy functional
\begin{equation*}
E^0(t)=\sum\limits_{|n|\in\{0,N(0)\}}\big(\frac{1}{2}\lV \d^{n} v \rV_{L^2}^2 +\int_0^t \lV \d^{n}L^{1/2} v \rV_{L^2}^2 ds +\frac{1}{2}\lV \d^{n}\Phi \rV_{L^2}^2\big).
\end{equation*}
Recall the system (\ref{Main_Sys}), we have
\begin{equation}       \label{dtE0}
\begin{aligned}
	\frac{d}{dt}E^0(t)=&\sum_{|n|\in\{0,N(0)\}}\int_{\R^3}\d^n v\cdot \d^n \U\P(-u\cdot\nab u-\d_j(\nab\phi\d_j\phi+\phi\nab u)+Err11+Err12)dx\\
	&- \sum_{|n|\in\{0,N(0)\}}\int_{\R^3} \d^n \d_t\phi\cdot \d^n (\d_t u\cdot\nab\phi+2u\cdot\nab\d_t\phi+u\cdot\nab(u\cdot\nab\phi)+Err2)dx.
\end{aligned}
\end{equation}
Next, we begin to estimate the right-hand side of (\ref{dtE0}) and finish the proof of the proposition.

\emph{Step 1: We prove the bound }
\begin{equation}         \label{E0_I1}
I_1:=\int_0^t \int_{\R^3} \d^n v\cdot \d^n\U\P(u\cdot\nab u+\nab(\nab\phi\nab\phi+\phi\nab u)+Err11+Err12)dxds\lesssim \ep_1^3.
\end{equation}
Using integration by parts, Sobolev embedding, (\ref{Main_Prop_Ass1}) and (\ref{Dec_Phi}), it follows that
\begin{align*}
I_1=&\int_0^t \int_{\R^3} \d^n v\cdot\d^n\U\P(u\cdot\nab u)-\d_j \d^nv \cdot\d^n\U\P(\nab\phi\d_j\phi+\phi\nab u+\d_j^{-1}(Err11+Err12))dxds\\
\lesssim & \int_0^t \lV v\rV_{H^n}\lV\nab u\rV_{H^n}^2+\lV\nab v\rV_{H^n}\lV\Phi\rV_{H^n}(\lV\Phi\rV_{L^{\infty}}+\lV\nab u\rV_{H^n})ds\\
\lesssim &\ep_1 \lV\nab u\rV_{L^2([0,t]:H^n)}^2+\ep_1\int_0^t \lV\nab u\rV_{H^n}\ep_1\<s\>^{-1+2\delta}ds
\lesssim  \ep_1^3.
\end{align*}

\emph{Step 2: We prove the bound}
\begin{equation}    \label{E0_I2}
I_2:=\int_0^t\int_{\R^3} \d^n \d_s\phi\cdot \d^n(\d_s u\cdot\nab \phi+2u\cdot\nab \d_s\phi)dxds\lesssim \ep_1^3.
\end{equation}
Integration by parts in time gives
\begin{equation*}
I_2 =\int_{\R^3}\d^n \d_s\phi\cdot \d^n(u\cdot\nab\phi)dx\big|_0^t+\int_0^t\int_{\R^3} \d^n \d_s\phi\cdot \d^n( u\cdot\nab \d_s\phi)dxds-\int_0^t\int_{\R^3} \d^n \d_s^2\phi\cdot \d^n( u\cdot\nab \phi)dxds.
\end{equation*}
By (\ref{Dec_Phi}), (\ref{Dec_d2v_Bet}) and $\div u=0$, the first two terms can be estimated by
\begin{align*}
&\int_{\R^3}\d^n \d_s\phi\cdot \d^n(u\cdot\nab\phi)dx\big|_0^t+\int_0^t\int_{\R^3} \d^n \d_s\phi\cdot \d^n( u\cdot\nab \d_s\phi)dxds\\
\lesssim& \lV\Phi\rV_{H^n}^2\lV u\rV_{H^n}+\int_0^t\int_{\R^3} \d^n \d_s\phi\cdot \sum_{n_1+n_2=n,n_2<n}(\d^{n_1} u\cdot\nab \d^{n_2}\d_s\phi)dxds\\
\lesssim & \ep_1^3+\int_0^t\lV\Phi\rV_{H^n}(\lV\<\nab\>^{|n|/2}\nab u\rV_{L^{\infty}}\lV\Phi\rV_{H^n}+\lV\nab u\rV_{H^n}\lV\<\nab\>^{|n|/2}\Phi\rV_{L^{\infty}})ds\\
\lesssim &\ep_1^3+\ep_1\int_0^t \big(\ep_1^2\<s\>^{-5/4+4\delta}+\lV\nab u\rV_{H^n}\<s\>^{-1+2\delta}\big)ds\\
\lesssim & \ep_1^3 .
\end{align*}
For the third term, it follows from (\ref{Main_Sys}) that
\begin{align*}
&\int_0^t\int_{\R^3} \d^n \d_s^2\phi\cdot \d^n( u\cdot\nab \phi)dxds\\
= & \int_0^t\int_{\R^3} \d^n \Delta\phi\cdot \d^n(u\cdot\nab\phi) dxds-\int_0^t\int_{\R^3} \d^n \d_s(u\cdot\nab\phi)\cdot \d^n(u\cdot\nab\phi) dxds\\
&-\int_0^t\int_{\R^3} \d^n (u\cdot\nab\d_s\phi)\cdot \d^n(u\cdot\nab\phi) dxds-\int_0^t\int_{\R^3} \d^n (u\cdot\nab(u\cdot\nab\phi))\cdot \d^n(u\cdot\nab\phi) dxds\\
&+\int_0^t\int_{\R^3} \d^n Err2\cdot \d^n(u\cdot\nab\phi) dxds\\
=&I_{21}+I_{22}+I_{23}+I_{24}+I_{25}.
\end{align*}
By integration by parts in $x$, (\ref{Main_Prop_Ass1}), (\ref{Dec_d2v_Bet}) and (\ref{Dec_Phi}), we have
\begin{align*}
I_{21}=&\int_0^t\int_{\R^3} -\d^n \d_i\phi\cdot \d^n(u\cdot\nab\d_i\phi+\d_i u\cdot\nab\phi)dxds\\
= & \int_0^t\int_{\R^3} -\sum_{n_1+n_2=n,n_2<n}\d^n\d_i\phi\cdot(\d^{n_1}u\cdot\nab\d^{n_2}\d_i\phi)-\d^n\d_i\phi\cdot\d^n(\d_i u\cdot\nab\phi)dxds\\
\lesssim & \int_0^t \lV\d^n\d_i\phi\rV_{L^2}(\lV\<\nab\>^{|n|/2}\nab u\rV_{L^{\infty}}\lV\Phi\rV_{H^n}+\lV\nab u\rV_{H^n}\lV\<\nab\>^{|n|/2}\Phi\rV_{L^{\infty}})ds\\
\lesssim &\ep_1^3.
\end{align*}
Similarly, using $\text{div}u=0$, (\ref{Main_Prop_Ass1}), (\ref{Dec_d2v_Bet}) and (\ref{Dec_Phi}) it's also easy to obtain
\begin{align*}
I_{22}+I_{24}=&-\frac{1}{2} \lV\d^n(u\cdot\nab\phi)\rV_{L^2}^2\Big|_0^t
-\int_0^t\int_{\R^3} \sum_{n_2<n}\d^{n_1}u\cdot\nab\d^{n_2}(u\cdot\nab\phi) \cdot\d^n(u\cdot\nab\phi)dxds\\
\lesssim & \sup_{s\in[0,t]}\lV u(s)\rV_{H^n}^2\lV\Phi(s)\rV_{H^n}^2+\int_0^t \lV  u\rV_{H^n}\lV \nab u\rV_{H^n}^2\lV\Phi\rV_{H^n}^2ds
\lesssim  \ep_1^4.
\end{align*}
and
\begin{equation*}
I_{25}\lesssim \int_0^t \lV\Phi\rV_{H^n}^3(1+\lV u\rV_{H^n}^2)(\lV\<\nab\>^{|n|/2}\nab u\rV_{L^{\infty}}\lV\Phi\rV_{H^n}+\lV\nab u\rV_{H^n}\lV\<\nab\>^{|n|/2}\Phi\rV_{L^{\infty}})ds\lesssim \ep_1^5.
\end{equation*}
Finally, by integration by parts, (\ref{Main_Prop_Ass1}) and (\ref{Dec_dtv_VF-bett}) we have
\begin{align*}
I_{23}=& -\int_0^t\int_{\R^3} \d^n(u\cdot\nab \d_s\phi)\cdot (u\cdot\nab\d^n\phi)+\sum_{n_2<n}\d^n(u\cdot\nab \d_s\phi)\cdot (\d^{n_1}u\cdot\nab\d^{n_2}\phi)dxds\\
=&-\int_0^t\int_{\R^3} \frac{1}{2}\d_s(u\cdot\nab\d^n\phi)^2-(\d_s u\cdot\nab\d^n\phi) \cdot (u\cdot\nab\d^n\phi)+\sum_{n_1+n_2=n,n_2<n}(\d^{n_1}u\cdot\nab\d^{n_2}\d_s\phi)\cdot(u\cdot\nab\d^n\phi)\\
&-\sum_{n_1+n_2=n,n_2<n}\d^{n-1}(u\cdot\nab\d_s\phi)\cdot\d (\d^{n_1}u\cdot\nab\d^{n_2}\phi)dxds\\
\lesssim & \sup_{s\in[0,t]}\lV u(s)\rV_{H^{|n|}}^2\lV \Phi(s)\rV_{H^{|n|}}^2+\int_0^t \big(\lV\d_s u\rV_{L^{\infty}}\lV \nab u\rV_{H^{|n|}}\lV\Phi\rV_{H^{|n|}}^2+\lV \nab u\rV_{H^{|n|}}^2\lV \Phi\rV_{H^{|n|}}^2\big)ds\\
\lesssim &\ep_1^4.
\end{align*}
Hence, the desired bound (\ref{E0_I2}) follows.

\emph{Step 3: We prove the bound}
\begin{equation}     \label{E0_I4}
I_4:=\int_0^t\int_{\R^3}\d^n\d_s\phi\cdot\d^n(u\cdot\nab(u\cdot\nab\phi))dxds\lesssim \ep_1^4.
\end{equation}
By integration by parts, we have
\begin{align*}
I_4=&\int_0^t\int_{\R^3} \d^n\d_s\phi\cdot(u\cdot\nab(u\cdot\nab\d^n\phi))+\sum_{n_1+n_2+n_3=n,n_3<n}\d^n\d_s\phi\cdot(\d^{n_1}u\cdot\nab(\d^{n_2}u\cdot\nab\d^{n_3}\phi))dxds\\
=:&I_{41}+I_{42}.
\end{align*}
The term $I_{42}$ can be estimated directly using (\ref{Main_Prop_Ass1}). From $\div u=0$, $I_{41}$ can be further rewritten as
\begin{align*}
I_{41}=&-\int_0^t\int_{\R^3} (u\cdot\nab\d^n\d_s\phi) \cdot (u\cdot\nab\d^n\phi)dxds\\
=& -\int_0^t\int_{\R^3} \frac{1}{2}\d_s (u\cdot\nab\d^n\phi)^2-(\d_s u\cdot\nab\d^n\phi)\cdot (u\cdot\nab\d^n\phi)dxds.
\end{align*}
From this and (\ref{Dec_dtv_VF-bett}) we obtain
\begin{equation*}
I_{41}\lesssim \sup_{s\in[0,t]} \lV u(s)\rV_{H^{|n|}}^2\lV\Phi(s)\rV_{H^{|n|}}^2+\int_0^t \lV\d_s u\rV_{L^{\infty}}\lV\nab u\rV_{H^{|n|}}\lV\Phi\rV_{H^{|n|}}^2ds\lesssim \ep_1^4.
\end{equation*}
The bound (\ref{E0_I4}) is obtained.

\emph{Step 4: We prove the bound}
\begin{equation}      \label{E0_I5}
I_5:=\int_0^t\int_{\R^3}\d^n\d_s\phi \cdot \d^n Err2(s)dxds\lesssim \ep_1^4.
\end{equation}
In fact, this is a consequence of
\begin{align*}
\lV\d^n Err2(s)\rV_{L^2}\lesssim \lV\Phi\rV_{H^{|n|}} \lV\Phi\rV_{L^{\infty}}\lV |\nab|^{-1}\Phi\rV_{L^{\infty}}\lesssim \ep_1^3\<s\>^{-5/4},
\end{align*}
which is given by  (\ref{Dec_Phi/nab}).

From (\ref{E0_I1})-(\ref{E0_I5}) and the assumption (\ref{MainAss_ini}), we have
\begin{equation*}
E^0(t)=E^0(0)+\int_0^t \d_s E^0(s)ds \lesssim \ep_0^2+\ep_1^3\lesssim \ep_0^2,
\end{equation*}
which completes the proof of the Proposition.

\end{proof}

\subsection{The bound on $v^{(a)}$ and $\Phi^{(a)}$}

\subsubsection{The bound on $v^{(a)}$}

\begin{proposition}     \label{Prop_va}
	With the notations and hypothesis in Proposition $\ref{Main_Prop}$, for any $t\in[0,T]$, $1\leq |a|\leq N_1$,
	\begin{equation}       \label{Ev}
	\lV   v^{(a)}\rV_{H^{N(a)}}^2+\int_0^t \lV \nab v^{(a)}\rV_{H^{N(a)}}^2 ds\lesssim \ep_0^2.
	\end{equation}
\end{proposition}
\begin{proof}
We prove the bound (\ref{Ev}) by induction. From Proposition \ref{Prop_Ene_Sob}, we assume that
\begin{equation}       \label{Ev_Ass}
\lV   v^{(c)}\rV_{H^{N(|c|)}}^2+\int_0^t \lV \nab v^{(c)}\rV_{H^{N(|c|)}}^2 ds\lesssim \ep_0^2,\ \mathrm{for}\ |c|< |a|.
\end{equation}
Define the energy functional
\begin{equation*}
E^a_v(t):=\sum\limits_{|n|\in\{0,N(a)\}}\big(\frac{1}{2}\lV \d^n v^{(a)}\rV_{L^2}^2+\int_0^t \lV \d^n L^{1/2} v^{(a)}\rV_{L^2}^2 ds\big).
\end{equation*}
It follows from (\ref{Main_Sys_VecFie}) that
\begin{align*}
\frac{d}{dt}E^a_v(t)=&-\sum_{|n|\in\{0,N(a)\}}\int_{\R^3}\d^n v^{(a)}\cdot\d^n\U\P L_1(u) dx\\
&-\sum_{|n|\in\{0,N(a)\}}\sum_{b+c=a}C_{a}^b\int_{\R^3}\d^n v^{(a)}\cdot\d^n\U\P(u^{(b)}\cdot\nab u^{(c)})dx\\
&+\sum_{|n|\in\{0,N(a)\}}\int_{\R^3}\d^n v^{(a)}\cdot\d^n\U\P(S+1)^{a_1}\Gamma^{a'}(\nab(\phi\nab u)-Err11)dx\\
&-\sum_{|n|\in\{0,N(a)\}}\sum_{b+c=a}C_a^b\int_{\R^3}\d^n v^{(a)}\cdot\d^n\U\P\d_j(\nab(S-1)^{b_1}\Gamma^{b'}\phi\cdot\d_j (S-1)^{c_1}\Gamma^{c'}\phi)dx\\
&+\sum_{|n|\in\{0,N(a)\}}\int_{\R^3}\d^n v^{(a)}\cdot\d^n\U\P(S+1)^{a_1}\Gamma^{a'}Err12dx\\
=&I^a_1+I^a_2+I^a_3+I^a_4+I^a_5.
\end{align*}

For $I^a_1$, by integration by parts, H\"{o}lder and (\ref{Ev_Ass}), we have
\begin{equation*}       \label{Ev_I1}
\begin{aligned}
\int_0^t I^a_1(s)ds =& \sum_{|n|\in\{0,N(a)\}} \int_0^t\int_{\R^3} \d^{n+1}  v^{(a)} \cdot\d^{n-1}\U\P L_1(u)dxds\\
\leq &\sum\limits_{|n|\in\{0,N(a)\}} \frac{1}{2}\int_0^t \big(\lV \d^n\nab v^{(a)}\rV_{L^2}^2 +C\sum\limits_{l=0}^{|a|-1}\lV \d^n\nab u^{(l)}\rV_{L^2}^2\big) ds\\
\leq & \frac{1}{2}\sum\limits_{|n|\in\{0,N(a)\}} \int_0^t \lV \d^n\nab v^{(a)}\rV_{L^2}^2ds +C\ep_0^2.
\end{aligned}
\end{equation*}

For $I^a_2$, from (\ref{Main_Prop_Ass1}) and H\"{o}lder inequality we have
\begin{align*}      \label{Ev_I2}
\int_0^t I_2^a(s)ds\lesssim  \int_0^t \lV v^{(a)} \rV_{H^{N(a)}}\lV \nab u^{(b)}\rV_{H^{N(a)}}\lV \nab u^{(c)}\rV_{H^{N(a)}}  ds\lesssim \ep_1^3.
\end{align*}

In order to estimate $I^a_3$, it suffices to prove that
\begin{equation}   \label{Ja3}
\sum_{|n|\in\{0,N(a)\}}\int_0^t \int_{\R^3}\d^n v^{(a)}\cdot \d^n\nab\U\P(\phi^{(b)}\nab u^{(c)}) dxds\lesssim \ep_1^3,\ \mathrm{for}\ b+c=a.
\end{equation}

If $|a|\leq N_1-2$, by (\ref{Hardy_2}), (\ref{Dec_Phi/nab}) and (\ref{Dec_d2v_Bet}) we have
\begin{align*}
\lV\phi^{(b)}\nab u^{(c)}\rV_{H^{N(a)}}\lesssim & \lV\phi^{(b)}\rV_{L^{\infty}}\lV\nab u^{(c)}\rV_{H^{N(a)}}+\lV\phi^{(b)}\rV_{\dot{H}^{N(a)}}\lV\nab u^{(c)}\rV_{L^{\infty}}\\
\lesssim &\ep_1\<s\>^{-1/2+3\delta}\lV\nab u^{(c)}\rV_{H^{N(a)}}+\ep_1^2\<s\>^{-5/4+(5+|c|)\delta};
\end{align*}
If $N_1-1\leq |a|\leq N_1, |b|\leq |c|$, by (\ref{Hardy_1}) and (\ref{Dec_Phi/nab}) we have
\begin{align*}
\lV\phi^{(b)}\nab u^{(c)}\rV_{H^{N(a)}}\lesssim (\lV\phi^{(b)}\rV_{L^{\infty}}+\sum_k 2^{N(a)k}\lV P_k\phi^{(b)}\rV_{L^{\infty}})\lV\nab u^{(c)}\rV_{H^{N(a)}}\lesssim \ep_1\<s\>^{-1/2+3\delta}\lV\nab u^{(c)}\rV_{H^{N(a)}};
\end{align*}
And if $N_1-1\leq |a|\leq N_1, |b|> |c|$, by (\ref{Dec_d2v_Bet}) and (\ref{Main_Prop_Ass1}) we have
\begin{align*}
\lV\phi^{(b)}\nab u^{(c)}\rV_{H^{N(a)}}\lesssim & \lV \phi^{(b)}\rV_{\dot{H}^{N(a)}}\lV\nab u^{(c)}\rV_{\Lf}+\lV\phi^{(b)}\rV_{L^7} \lV \nab u^{(c)}\rV_{W^{N(a)+1,14/5}}\\
\lesssim &\lV\Phi^{(b)}\rV_{H^{N(a)}} \lV\nab u^{(c)}\rV_{W^{N(a)+1,14/5}}\\
\lesssim &\ep_1 \<s\>^{2\delta} \lV \nab u^{(c)}\rV_{H^{N(a)+2}}^{5/7} \lV \nab u^{(c)}\rV_{W^{N(a)+2,3/2\delta}}^{2/7}\\
\lesssim & \ep_1^{9/7}\lV \nab u^{(c)}\rV_{H^{N(a)+2}}^{5/7} \<s\>^{-5/14+4\delta}.
\end{align*}
Then using integration by parts, the above three bounds, H\"{o}lder and (\ref{Main_Prop_Ass1}), we obtain the bound (\ref{Ja3}), and hence conclude the estimate of $I^a_3$.

Finally, in order for $I^a_4$ and $I^a_5$ it suffices to prove that
\begin{equation}   \label{Ja4}
\sum_{|n|\in\{0,N(a)\}}\int_0^t \int_{\R^3}\d^n v^{(a)}\cdot \d^n\nab\U\P(\nab\phi^{(b)}\nab \phi^{(c)}) dxds\lesssim \ep_1^3,\ \ \mathrm{for}\ |b|+|c|\leq |a|,
\end{equation}
which is obtained by integration by parts, (\ref{phi^2-Hna}), H\"{o}lder and (\ref{Main_Prop_Ass1}). Hence, this concludes the estimates of $I^a_4$ and $I^a_5$.

As a consequence of the above estimates of $I^a_1,\cdots,I^a_5$ and (\ref{MainAss_ini}), we have
\begin{equation*}
E^a_v(t)=E^a_v(0)+\int_0^t \d_s E^a_v(s)ds\leq \frac{1}{2}\sum_{|n|\in\{0,N(a)\}}\int_0^t \lV\d^n\nab v^{(a)}\rV_{L^2}^2ds+C\ep_0^2,
\end{equation*}
which implies
\begin{equation*}
\sum_{|n|\in\{0,N(a)\}}(\lV\d^n v^{(a)}\rV_{L^2}^2+\int_0^t \lV\d^n\nab v^{(a)}\rV_{L^2}^2ds) \lesssim \ep_0^2.
\end{equation*}
This completes the proof of the Proposition.
\end{proof}

\subsubsection{The bound on $\Phi^{(a)}$}

\begin{proposition}     \label{Prop_Phia}
	With the notation and hypothesis in Proposition $\ref{Main_Prop}$, for any $t\in[0,T]$, $1\leq |a|\leq N_1$,
	\begin{equation}       \label{Ephi}
	\lV   \Phi^{(a)}(t)\rV_{H^{N(a)}}\lesssim \ep_0\<t\>^{\ka(a)\delta}.
	\end{equation}
\end{proposition}
\begin{proof}
	Define the energy functional
	\begin{equation*}
	E^a_{\phi}(t):=\frac{1}{2}\sum_{|n|\in\{0,N(a)\}}(\lV\d^n\d_t\phi^{(a)}\rV_{L^2}^2+\lV\d^n\nab\phi^{(a)}\rV_{L^2}^2).
	\end{equation*}
	Using the $\phi^{(a)}$-equation in (\ref{Main_Sys_VecFie}) we calculate
	\begin{align*}
	\frac{d}{dt}E^a_{\phi}(t)=&-2\sum_{b+c=a}C_a^b\sum_{|n|\in\{0,N(a)\}}\int_{\R^3}\d^n\d_t\phi^{(a)}\cdot \d^n(u^{(b)}\cdot\nab\d_t\phi^{(c)})dx\\
	&-\sum_{b+c=a}C_a^b\sum_{|n|\in\{0,N(a)\}}\int_{\R^3}\d^n\d_t\phi^{(a)}\cdot \d^n(\d_t u^{(b)}\cdot\nab\phi^{(c)})dx\\
	&-\sum_{b+c+e=a}C_a^{b,c}\sum_{|n|\in\{0,N(a)\}}\int_{\R^3}\d^n\d_t\phi^{(a)}\cdot \d^n(u^{(b)}\cdot\nab(u^{(c)}\cdot\nab\phi^{(e)}))dx\\
	&+\sum_{|n|\in\{0,N(a)\}}\int_{\R^3}\d^n\d_t\phi^{(a)}\cdot\d^n(S+1)^{a_1}\Gamma^{a'}Err2\  dx\\
	=:& -2II^a_1-II^a_2-II^a_3+II^a_4.
	\end{align*}
	
	\emph{Step 1: We prove the bound }
	\begin{equation}          \label{Ephi_II_1}
	\int_0^t II^a_1(s)ds\lesssim \ep_1^3\<t\>^{2\ka(a)\delta}.
	\end{equation}

	By $\div u=0$ and integration by parts, $II^{a}_1$ can be rewritten as
	\begin{align*}
	II^{a}_1(s)=&\sum_{n_1+n_2=n,|n_2|<|n|=N(a)}\int_{\R^3} \d^{n}\d_s\phi^{(a)}\cdot(\d^{n_1}u\cdot\nab\d^{n_2}\d_s\phi^{(a)})dx\\
	&+\sum_{b+c=a,|b|\geq 1}C_a^b \sum_{|n|\in\{0,N(a)\}}\int_{\R^3}\d^n \d_s\phi^{(a)}\cdot\d^n(u^{(b)}\cdot\nab\d_s\phi^{(c)})dx\\
	=:& II^{a}_{11}+II^{a}_{12}.
	\end{align*}
	First, we estimate $II^{a}_{11}$. When $|a|\geq 2$, from (\ref{Dec_d2v_Bet}) and Sobolev embedding we have
	\begin{align*}
	\int_0^t II^{a}_{11}(s)ds\lesssim &\int_0^t \lV\d_s\phi^{(a)}\rV_{H^{N(a)}}^2 \lV \<\nab\>^{N(a)-1}\nab u\rV_{L^{3/2\delta}}ds\\
	\lesssim & \int_0^t \ep_1^2 \<s\>^{2\ka(a)\delta} \lV \nab v\rV_{W^{N(a)-1,{3/2\delta}}}ds
	\lesssim \int_0^t \ep_1^3\<s\>^{-5/4+(2\ka(a)+4)\delta}ds\lesssim \ep_1^3.
	\end{align*}
	When $|a|=1$, by (\ref{Dec_d2v_Bet}) and (\ref{Dec_Phi}) we have
	\begin{align*}
	\int_0^t II^{a}_{11}(s)ds\lesssim &\int_0^t \lV \d_s\phi^{(a)}\rV_{H^{N(a)}}(\lV\<\nab\>^{N(a)/2}\nab u\rV_{L^{\infty}}\lV\d_s\phi^{(a)}\rV_{H^{N(a)}}+\lV\nab u\rV_{H^{N(a)}}\lV \<\nab\>^{N(a)/2}\d_s\phi^{(a)}\rV_{L^{\infty}})ds\\
	\lesssim & \int_0^t \ep_1 \<s\>^{\delta}(\ep_1^2\<s\>^{-5/4+4\delta}+\ep_1\<s\>^{-1+2\delta}\lV \nab u\rV_{H^{N(a)}})ds\\
	\lesssim &\ep_1^3.
	\end{align*}
	Second, we estimate $II^a_{12}$ when $|b|\geq N_1-1$. (\ref{Dec_Phi}) and (\ref{Main_Prop_Ass1}) imply that
	\begin{align*}
	\int_0^t II^{a}_{12}(s)ds \lesssim &\int_0^t \lV \d_s\phi^{(a)}\rV_{H^{N(a)}} \lV u^{(b)}\rV_{H^{N(a)}}\sum_k 2^{N(a)k^++k^+}\lV \d_s P_k\phi^{(c)}\rV_{L^{\infty}}ds\\
	\lesssim & \int_0^t \ep_1^3 \<s\>^{-1+4\delta}ds\lesssim \ep_1^3 \<t\>^{2\ka(a)\delta}.
	\end{align*}
	 Finally it remains to prove that when $1\leq|b|\leq N_1-2$, $|n|\in\{0,N(a)\}$
	\begin{equation}            \label{Lem4.4}
	\int_0^t\int_{\R^3} \d^n \d_s\phi^{(a)}\cdot\d^n (u^{(b)}\cdot\nab\d_s\phi^{(c)})dxds\lesssim \ep_1^3 \<t\>^{2\ka(a)\delta}.
	\end{equation}	
	
	We decompose dyadically in frequency and rewrite the functions $\d_s\phi^{(a)}$, $\d_s\phi^{(c)}$ in terms of the variables $\Psi^{(a)}$, $\Psi^{(c)}$, it suffices to estimate
	\begin{align}  \label{E4_bound}
	II_{k,k_1,k_2;m}:=\int_1^t\int_{\R^6}e^{-is(|\xi|-|\eta|)}m(\xi,\eta) \overline{ \widehat{P_k\Psi^{(a)}}}(\xi) \widehat{P_{k_1}u^{(b)}}(\xi-\eta)\widehat{P_{k_2}\Psi^{(c)}}(\eta)d\xi d\eta ds,
	\end{align}
	for any $t\in[1,T]$, $m(\xi,\eta):=(1+\xi^{2n}) \eta$, $|n|=N(a),b+c=a$. By $\div u^{(b)}=0$ and $[\nab,P_k]=0$, we have
	\begin{align}       \label{eta.u}
	\eta\cdot\widehat{P_{k_1}u^{(b)}}(\xi-\eta)=(\eta-\xi+\xi)\cdot \widehat{P_{k_1}u^{(b)}}(\xi-\eta)=\xi\cdot \widehat{P_{k_1}u^{(b)}}(\xi-\eta).
	\end{align}
	Then denote
	\begin{equation*}
	m_{hh}:=(1+\xi^{2n})\xi,
	\end{equation*}
	we obtain from (\ref{eta.u})
	\begin{align*}
	m(\xi,\eta)\cdot\widehat{P_{k_1}u^{(b)}}(\xi-\eta)=&\mathbf{1}_{k_2\leq k+5}\cdot m(\xi,\eta)\cdot\widehat{P_{k_1}u^{(b)}}(\xi-\eta)+\mathbf{1}_{k_2> k+5}\cdot m(\xi,\eta)\cdot\widehat{P_{k_1}u^{(b)}}(\xi-\eta)\\
	=&\mathbf{1}_{k_2\leq k+5}\cdot m(\xi,\eta)\cdot\widehat{P_{k_1}u^{(b)}}(\xi-\eta)+\mathbf{1}_{k_2> k+5}\cdot m_{hh}(\xi,\eta)\cdot\widehat{P_{k_1}u^{(b)}}(\xi-\eta)\\
	=:&m^{*}(\xi,\eta)\cdot \widehat{P_{k_1}u^{(b)}}(\xi-\eta).
	\end{align*}

	\emph{Case 1: Low-high, i.e $k_2\in[k-5,k+5]$.}
	
	Integrating by parts in $\eta$, we have
	\begin{align*}
	II_{k,k_1,k_2;m}
	=&\int_1^t\int_{\R^6} s^{-1}\d_{\eta_j}(\frac{\eta_j}{|\eta|}m(\xi,\eta))  \overline{\widehat{P_k\Phi^{(a)}}}(\xi) \widehat{P_{k_1}u^{(b)}}(\xi-\eta)\widehat{P_{k_2}\Phi^{(c)}}(\eta)d\xi d\eta ds\\
	&+\int_1^t\int_{\R^6} s^{-1}\frac{\eta_j}{|\eta|}m(\xi,\eta)  \overline{\widehat{P_k\Phi^{(a)}}}(\xi) \d_{\eta_j}\widehat{P_{k_1}u^{(b)}}(\xi-\eta)\widehat{P_{k_2}\Phi^{(c)}}(\eta)d\xi d\eta ds\\
	&+\int_1^t\int_{\R^6} s^{-1}\frac{\eta_j}{|\eta|}m(\xi,\eta)  \overline{\widehat{P_k\Phi^{(a)}}}(\xi) \widehat{P_{k_1}u^{(b)}}(\xi-\eta)e^{is|\eta|}\d_{\eta_j}\widehat{P_{k_2}\Psi^{(c)}}(\eta)d\xi d\eta ds.
	\end{align*}
	Then estimating $u^{(b)}$ in ${L^{\infty}}$ and the other two terms in $L^2$, it follows from Lemma $\ref{Sym_Pre_lemma}$ that
	\begin{align*}
	&\sum_{k,k_1}\sum_{k_2\in[k-5,k+5]}  II_{k,k_1,k_2;m}\\
	\lesssim &\int_1^t s^{-1}\sum_{|k_2-k|<5}\sum_{k_1<k+6} 2^{2N(a)k^+}\lV P_k\Phi^{(a)}\rV_{L^2} \Big(\lV P_{k_1}u^{(b)}\rV_{L^{\infty}} \lV P_{k_2}\Phi^{(c)}\rV_{L^2} \\
	&+2^k\lV \mathcal{F}^{-1}(\nab_{\xi}\widehat{P_{k_1}u^{(b)}}(\xi))\rV_{L^{\infty}} \lV P_{k_2}\Phi^{(c)}\rV_{L^2} +2^k \lV P_{k_1}u^{(b)}\rV_{L^{\infty}} \lV \nab_{\xi} \widehat{P_{k_2}\Psi^{(c)}}(\xi)\rV_{L^2} \Big)ds\\
	\lesssim & \int_1^t s^{-1}\lV\Phi^{(a)}\rV_{H^{N(a)}}\Big(\lV \nab u^{(b)}\rV_{H^1}\lV \Phi^{(c)}\rV_{H^{N(a)}}+\sum_{k_1}\lV \<r\>P_{k_1}u^{(b)}\rV_{L^{\infty}}\lV \Phi^{(c)}\rV_{H^{N(a)+1}}\\
	&+\lV\nab u^{(b)}\rV_{H^1}\lV\FF^{-1}(\xi|\d_{\xi}\widehat{\Psi^{(c)}})\rV_{H^{N(a)}}\Big)ds.
	\end{align*}
	Then by (\ref{Main_Prop_Ass1}), (\ref{Main_Prop_Ass2}) and (\ref{rv}) we get
	\begin{align*}
	\sum_{k,k_1}\sum_{k_2\in[k-5,k+5]}  II_{k,k_1,k_2;m}\lesssim & \int_1^t  \ep_1\<s\>^{-1+\ka(a)\delta}(\lV \nab u^{(b)}\rV_{H^1}\ep_1\<s\>^{\ka(|c|+1)\delta}+\ep_1^2\<s\>^{\ka(c)\delta})ds\\
	\lesssim & \ep_1^3 \<t\>^{2\ka(a)\delta}.
	\end{align*}

	\emph{Case 2: High-low and high-high, i.e $k_2\in(-\infty,k-5)\cup (k+5,\infty)$.} 
	
	Since the phase $|\xi|-|\eta|$ in (\ref{E4_bound}) doesn't equal to zero, by integration by parts in times, it suffices to prove that
	\begin{equation}           \label{E4_k2=!k}
	\sum\limits_{k,k_1,k_2}\big(II^{0}_{k,k_1,k_2}-\int_1^t II^{1}_{k,k_1,k_2}+II^{2}_{k,k_1,k_2}+II^{3}_{k,k_1,k_2}ds\big) \lesssim \epsilon_1^3\<t\>^{2\ka(a)\delta},
	\end{equation}
	where
	\begin{gather*}
	II^{0}_{k,k_1,k_2}:=\mathcal{Q}[\tilde{m};P_k\Psi^{(a)},P_{k_1}u^{(b)},P_{k_2}\Psi^{(c)}]\big|_1^t,\ \ 
	II^{1}_{k,k_1,k_2}:=\mathcal{Q}[\tilde{m};\d_sP_k\Psi^{(a)},P_{k_1}u^{(b)},P_{k_2}\Psi^{(c)}],\\
	II^{2}_{k,k_1,k_2}:=\mathcal{Q}[\tilde{m};P_k\Psi^{(a)},\d_sP_{k_1}u^{(b)},P_{k_2}\Psi^{(c)}],\ \ 
	II^{3}_{k,k_1,k_2}:=\mathcal{Q}[\tilde{m};P_k\Psi^{(a)},P_{k_1}u^{(b)},\d_s P_{k_2}\Psi^{(c)}],
	\end{gather*}
	and
	\begin{gather*}
	\tilde{m}(\xi,\eta):=m^{*}(\xi,\eta)/(|\xi|-|\eta|),\\
	\mathcal{Q}[\tilde{m};f,g,h]:=\int_{\R^6} ie^{-is(|\xi|-|\eta|)}\tilde{m}(\xi,\eta)\overline{\widehat{f}(\xi)}\widehat{g}(\xi-\eta)\widehat{h}(\eta)d\xi d\eta.
	\end{gather*}
	
	To prove the bound (\ref{E4_k2=!k}), we need the $S^{\infty}$ norm of symbol $\tilde{m}(\xi,\eta)$. By Lemma $\ref{ContSymb}$, we have
	\begin{equation} \label{S_tildm}
	\lV \tilde{m}(\xi,\eta)\rV_{S^{\infty}_{k,k_1,k_2}}\lesssim 2^{2N(a)k^++\min(k-k_2,k_2-k)}.
	\end{equation}
	
	\emph{Case 2.1: The contribution of $II^0_{k,k_1,k_2}$.} 
	
	We estimate the lowest frequency factor $P_k\Psi^{(a)}$ or $P_{k_2}\Psi^{(c)}$ in $L^{\infty}$ and the other two factors in $L^2$, using (\ref{S_tildm}), H\"{o}lder and (\ref{Main_Prop_Ass1}) we have
	\begin{align*}
	\sum_{k,k_1}\sum_{k_2\in(-\infty,k-5)\cup(k+5,\infty)}II^0_{k,k_1,k_2}\lesssim & \sum_{|k-k_1|<4}\sum\limits_{k_2<k-5} 2^{2N(a)k^++k_2-k}\lV P_k\Psi^{(a)}\rV_{L^2}\lV P_{k_1}u^{(b)}\rV_{L^2} \lV P_{k_2}\Psi^{(c)}\rV_{L^{\infty}} \\
	&+\sum\limits_{k}\sum_{k_2>k+5,|k_1-k_2|<4} 2^{2N(a)k^++k-k_2}\lV P_k\Psi^{(a)}\rV_{L^{\infty}}\lV P_{k_1}u^{(b)}\rV_{L^2} \lV P_{k_2}\Psi^{(c)}\rV_{L^2}\\
	\lesssim & \sum_{k,k_1;|k-k_1|<4}2^{2N(a)k^+}\lV P_k\Psi^{(a)}\rV_{L^2}\lV P_{k_1}u^{(b)}\rV_{L^2} \sum_{k_2}\lV P_{k_2}\Psi^{(c)}\rV_{L^{\infty}}\\
	&+ \sum_{k}\lV P_{k_2}\Psi^{(a)}\rV_{L^{\infty}} \sum_{k_1,k_2;|k_1-k_2|<4}2^{2N(a)k_2^+}\lV P_k\Psi^{(a)}\rV_{L^2}\lV P_{k_1}u^{(b)}\rV_{L^2}\\
	\lesssim &\lV\Phi^{(a)}\rV_{H^{N(a)}}\lV u^{(b)}\rV_{H^{N(a)}}\lV\Phi^{(c)}\rV_{H^{N(a)}}\lesssim \ep_1^3\<t\>^{2\ka(a)\delta}.
	\end{align*}
	
	\emph{Case 2.2: The contribution of $II^1_{k,k_1,k_2}$ and $II^3_{k,k_1,k_2}$.} 
	
	When $a=b, c=0,k_2<k-5$, by (\ref{S_tildm}), (\ref{H^n-1_dttphi^a}), (\ref{dt(Psi)_2norm_dec}) and (\ref{Dec_Phi}), one obtain
	\begin{align*}
	&\sum_{k,k_1}\sum_{k_2<k-5}II^1_{k,k_1,k_2}\\
	\lesssim & \sum_{|k-k_1|<4}\sum_{k_2<k-5} 2^{2N(a)k^++k_2-k} \lV \d_s P_k\Psi^{(a)}\rV_{L^2}\lV P_{k_1}u^{(a)}\rV_{L^2}  \lV  P_{k_2}\Phi\rV_{L^{\infty}}  \\
	\lesssim &\lV\d_s\Psi^{(a)}\rV_{H^{N(a)-2}}\lV \nab u^{(a)}\rV_{H^{N(a)}}\sum_{k_2}2^{k_2^+}\lV P_{k_2}\Phi\rV_{L^{\infty}}+\lV\d_s\Psi^{(a)}\rV_{2}\lV  u^{(a)}\rV_{2}\sum_{k_2}\lV P_{k_2}\Phi\rV_{L^{\infty}}\\
	\lesssim & \ep_1^2\<s\>^{-1+6\delta}\lV \nab u^{(a)}\rV_{H^{N(a)}}+\ep_1^4\<s\>^{-7/4+(|a|+3)\delta}.
	\end{align*}
	When $a=b,c=0,k_2>k+5$, from (\ref{S_tildm}), (\ref{dt(Psi)_2norm_dec}) and (\ref{Dec_Phi}) it follows that
	\begin{align*}
	&\sum_{k,k_1}\sum_{k_2>k+5}II^1_{k,k_1,k_2}\\
	\lesssim & \sum_k\sum_{k_2>k+5,|k_1-k_2|<4} 2^{2N(a)k^++k-k_2}\lV \d_s P_k\Psi^{(a)}\rV_{L^2}\lV P_{k_1}u^{(a)}\rV_{L^2}  \lV  P_{k_2}\Phi\rV_{L^{\infty}}\\
	\lesssim & \sum_{k<0}\sum_{k_2>k+5,|k_1-k_2|<4} 2^{k-k_2}\lV \d_s P_k\Psi^{(a)}\rV_{L^2}\lV P_{k_1}u^{(a)}\rV_{L^2}  \lV  P_{k_2}\Phi\rV_{L^{\infty}}\\
	&+\sum_{k\geq 0}\sum_{k_2>k+5,|k_1-k_2|<4} 2^{(N(a)-3)k+(N(a)+4)(k-k_2)+(N(a)+3)k_2}\lV \d_s P_k\Psi^{(a)}\rV_{L^2}\lV P_{k_1}u^{(a)}\rV_{L^2}  \lV  P_{k_2}\Phi\rV_{L^{\infty}}\\
	\lesssim & \lV \d_s\Psi^{(a)}\rV_{L^2} \lV u^{(a)}\rV_{L^2} \lV \Phi\rV_{L^{\infty}}+\lV \d_s \Psi^{(a)}\rV_{H^{N(a)-3}}\lV u^{(a)}\rV_{H^{N(a)}}\sup_{k}2^{3k}\lV P_k\Phi\rV_{L^{\infty}} \\
	\lesssim & \ep_1^3 \<s\>^{-7/4+(|a|+3)\delta}.
	\end{align*}
	Then the other cases, i.e. $|b|<|a|$ can be estimated, using (\ref{dt(Psi)_2norm_dec}) and (\ref{Main_Prop_Ass1}),
	\begin{align*}
	\sum_{k,k_1}\sum\limits_{k_2\in(-\infty,k-5)\cup(k+5,\infty)}II^1_{k,k_1,k_2}	\lesssim & \sum_{|k-k_1|<4}\sum_{k_2<k-5} 2^{2N(a)k^++k_2-k} \lV\d_sP_k\Psi^{(a)}\rV_{L^2}\lV P_{k_1}u^{(b)}\rV_{L^2}\lV P_{k_2}\Phi^{(c)}\rV_{L^{\infty}}\\
	&+\sum_k\sum_{k_2>k+5,|k_1-k_2|<4} 2^{2N(a)k^++k-k_2} \lV\d_sP_k\Psi^{(a)}\rV_{L^2}\lV P_{k_1}u^{(b)}\rV_{L^2}\lV P_{k_2}\Phi^{(c)}\rV_{L^{\infty}}\\
	\lesssim &\lV\d_s\Psi^{(a)}\rV_{H^{N(a)-3}}\lV\nab u^{(b)}\rV_{H^{N(b)}}\lV\Phi^{(c)}\rV_{H^{N(a)}}\\
	\lesssim & \ep_1^2 \<s\>^{-3/4+(|a|+3)\delta}\lV\nab u^{(b)}\rV_{H^{N(b)}}.
	\end{align*}
	Similarly, by (\ref{S_tildm}), (\ref{dt(Psi)_2norm_dec}) and (\ref{Main_Prop_Ass1}), we have
	\begin{align*}
	\sum_{k,k_1}\sum\limits_{k_2\in(-\infty,k-5)\cup(k+5,\infty)}II^3_{k,k_1,k_2}
	\lesssim &\lV\Psi^{(a)}\rV_{H^{N(a)}}\lV\nab u^{(b)}\rV_{H^{N(a)}}\lV\d_s\Psi^{(c)}\rV_{H^{N(c)-3}}\\
	\lesssim & \ep_1^2 \<s\>^{-3/4+(|a|+3)\delta}\lV\nab u^{(b)}\rV_{H^{N(a)}}.
	\end{align*}
	By H\"{o}lder, the bound (\ref{E4_k2=!k}) for $II^1_{k,k_1,k_2}$ and $II^3_{k,k_1,k_2}$ follows.
	
	\emph{Case 2.3: The contribution of $II^2_{k,k_1,k_2}$.} 
	
	The contribution of the triplets $(k,k_1,k_2)$ with $k_2>k+5,k_1\in[k_2-3,k_2+3]$ can be estimated using (\ref{Dec_dtv_VF}), i.e
	\begin{equation*}\label{I2_c1}
	\begin{aligned}
	\sum_{k,k_1}\sum_{k_2>k+5}II^2_{k,k_1,k_2}\lesssim & \sum_k\sum_{k_2>k+5,|k_1-k_2|<4} 2^{2N(a)k^++k-k_2}\lV P_k\Phi^{(a)}\rV_{L^2} \lV \d_s P_{k_1} u^{(b)}\rV_{L^{\infty}}\lV P_{k_2}\Phi^{(c)}\rV_{L^2}\\
	\lesssim & \lV \Phi^{(a)}\rV_{H^{N(a)}}\lV \d_s u^{(b)}\rV_{L^{\infty}}\lV \Phi^{(c)}\rV_{H^{N(a)}}\lesssim \ep_1^3 \<s\>^{-5/4+(8+|b|)\delta}.
	\end{aligned}
	\end{equation*}
	
	Then we consider $II^2_{k,k_1,k_2}$ when $k_2<k-5$. If $|c|\geq N_1-1$, we obtain from (\ref{S_tildm}), (\ref{Dec_dtv_VF}) and (\ref{Main_Prop_Ass1})
	\begin{equation*}       \label{I2_c2}
	\begin{aligned}
	\sum_{k,k_1}\sum_{k_2<k-5}II^2_{k,k_1,k_2}\lesssim & \sum_{|k-k_1|<4}\sum_{k_2<k-5} 2^{2N(a)k^++k_2-k} \lV P_k\Phi^{(a)}\rV_{L^2} \lV \d_s P_{k_1} u^{(b)}\rV_{L^{\infty}}\lV P_{k_2}\Phi^{(c)}\rV_{L^2}  \\
	\lesssim &\lV \Phi^{(a)}\rV_{H^{N(a)}}\sum_{k_1}2^{N(a)k_1^+}\lV P_{k_1}\d_s u^{(b)}\rV_{L^{\infty}}\lV \Phi^{(c)}\rV_{H^2}\lesssim \ep_1^3 \<s\>^{-5/4+(8+|b|)\delta}.
	\end{aligned}
	\end{equation*}
	If $|c|\leq N_1-2$, by (\ref{S_tildm}), (\ref{Main_Prop_Ass1}), (\ref{H^n_dtv}) and (\ref{Dec_Phi}), we have
	\begin{equation*}       \label{I2_c3}
	\begin{aligned}
	\sum_{k,k_1}\sum_{k_2<k-5}II^2_{k,k_1,k_2}\lesssim & \sum_{|k-k_1|<4}\sum_{k_2<k-5}2^{2N(a)k^++k_2-k} \lV P_k\Phi^{(a)}\rV_{L^2} \lV \d_s P_{k_1} u^{(b)}\rV_{L^2}\lV P_{k_2}\Phi^{(c)}\rV_{L^{\infty}}\\
	\lesssim &\lV \Phi^{(a)}\rV_{H^{N(a)}}\lV \d_s u^{(b)}\rV_{H^{N(a)-1}}\sum_{k_2}2^{k_2^+}\lV P_{k_2}\Phi^{(c)}\rV_{L^{\infty}}\\
	\lesssim &\ep_1^2 \<s\>^{-1+7\delta}\sum_{l\leq b}\lV \nab u^{(l)}\rV_{H^{N(b)}}+\ep_1^3\<s\>^{-2+10\delta}.
	\end{aligned}
	\end{equation*}
	These imply the bound (\ref{E4_k2=!k}) for $II^2_{k,k_1,k_2}$. Then the bound (\ref{E4_k2=!k}) is obtained, and hence we obtain the bound (\ref{Lem4.4}).

	\emph{Step 2: We prove the bound}
	\begin{equation}    \label{Ephi_II2}
	\int_0^t II^{a}_2(s)ds\lesssim \ep_1^3\<t\>^{2\ka(a)\delta}.
	\end{equation}
	
	Notice that when $|b|<|c|$,  (\ref{Dec_dtv_VF}) implies
	\begin{equation*}
	\lV\d_su^{(b)}\cdot\nab\phi^{(c)}\rV_{L^2} \lesssim \lV \d_s u^{(b)}\rV_{L^{\infty}}\lV\nab\phi^{(c)}\rV_{L^2} \lesssim \ep_1^2\<s\>^{-5/4+(6+|b|)\delta},
	\end{equation*}
	and when $|b|\geq |c|$,  (\ref{H^n_dtv}) and (\ref{Dec_Phi}) give
	\begin{equation*}
	\lV\d_su^{(b)}\cdot\nab\phi^{(c)}\rV_{L^2} \lesssim \lV \d_s u^{(b)}\rV_{L^2}\lV\nab\phi^{(c)}\rV_{L^{\infty}} \lesssim\sum_{|l|\leq |b|}\ep_1\<s\>^{-1+5\delta}\lV \nab v^{(l)}\rV_{H^{N(b)}}+\ep_1^3\<s\>^{-2+8\delta}.
	\end{equation*}
	Therefore,
	\begin{equation*}
	\sum_{b+c=a} C_a^b\int_0^t\int_{\R^3}\d_s\phi^{(a)}\cdot (\d_su^{(b)}\cdot\nab\phi^{(c)})dxds\lesssim \ep_1^3.
	\end{equation*}
	
	Now it suffices to prove that
	\begin{equation*}
    \int_0^t\widetilde{II}^{a}_{2}(s)ds:=\sum_{b+c=a}\sum_{|n|=N(a)}\int_0^t\int_{\R^3}\d^{n}\d_s \phi^{(a)}\cdot\d^{n}(\d_s u^{(b)}\cdot\nab\phi^{(c)})dxds\lesssim \ep_1^3\<t\>^{2\delta}.
	\end{equation*}
	$\widetilde{II}^{a}_2$ can be divided into three terms, i.e.
	\begin{align*}
	\widetilde{II}^{a}_2=&\sum_{|n|=N(a)}\int_{\R^3}\d^{n}\d_s \phi^{(a)}\cdot(\d^{n}\d_s u^{(a)}\cdot\nab\phi)dx\\
	&+\sum_{|n|=N(a)}\int_{\R^3}\d^{n}\d_s \phi^{(a)}\cdot\sum_{n_1+n_2=n,n_1<n}(\d^{n_1}\d_s u^{(a)}\cdot\nab\d^{n_2}\phi)dx\\
	&+\sum_{b+c=a,|b|<|a|}\sum_{|n|=N(a)}C_a^b\int_{\R^3} \d^n \d_s\phi^{(a)}\cdot\d^n(\d_s u^{(b)}\cdot\nab\phi^{(c)})dx\\
	=:& \tilde{II}^{a}_{21}+\tilde{II}^{a}_{22}+\tilde{II}^{a}_{23}.
	\end{align*}
	
	We estimate $\tilde{II}^{a}_{21}, \tilde{II}^{a}_{22},\tilde{II}^{a}_{23}$ respectively. Firstly, we consider the term $\tilde{II}^{a}_{21}$. Integrating by parts, we have
	\begin{align*}
	\tilde{II}^{a}_{21}=&\sum_{|n|=N(a)}\big[\d_s \int_{\R^3}\d^{n}\d_s \phi^{(a)}\cdot(\d^{n} u^{(a)}\cdot\nab\phi)dx +\int_{\R^3}\d^{n-1}\d^2_s \phi^{(a)}\cdot\d(\d^{n}u^{(a)}\cdot\nab\phi)dx\\
	&-\int_{\R^3}\d^{n}\d_s \phi^{(a)}\cdot(\d^{n}u^{(a)}\cdot\nab\d_s\phi)dx\big].
	\end{align*}
	Using (\ref{Dec_Phi}) and (\ref{H^n-1_dttphi^a}), it follows that
	\begin{equation}              \label{II*1}
	\begin{aligned}
	\int_0^t \tilde{II}^{a}_{21}(s)ds\lesssim & \sum_{|n|=N(a)}\big[\sup_{s\in[0,t]}\lV\d^n\d_s\phi^{(a)}\rV_{L^2}\lV\d^n u^{(a)}(s)\rV_{L^2}\lV\nab\phi(s)\rV_{H^2}\\
	&+\int_0^t \lV\d^{n-1}\d^2_s\phi^{(a)}\rV_{L^2}\lV\nab u^{(a)}\rV_{H^{N(a)}}\lV\<\nab\>^2\phi\rV_{L^{\infty}}ds\\
	&+\int_0^t \lV\d^{n}\d_s\phi^{(a)}\rV_{L^2}\lV\nab u^{(a)}\rV_{H^{N(a)}}\lV\nab\d_s\phi\rV_{L^{\infty}}ds\big]\\
	\lesssim & \ep_1^3\<t\>^{\ka(a)\delta}+\int_0^t \ep_1^2\<s\>^{-1+6\delta}(1+\sum_{|l|\leq|a|}\lV \nab u^{(l)}\rV_{H^{N(a)}})\lV \nab u^{(l)}\rV_{H^{N(a)}}ds\\
	&+\int_0^t \ep_1^2\<s\>^{\ka(a)\delta-1+2\delta}\lV \nab u^{(a)}\rV_{H^{N(a)}}ds\\
	\lesssim & \ep_1^3 \<t\>^{\ka(a)\delta}.
	\end{aligned}
	\end{equation}
	Second, we consider the term $\tilde{II}^{a}_{22}$. If $|a|\geq N_1-1$, it follows from (\ref{Dec_Phi}) and (\ref{H^n_dtv}) that
	\begin{align*}
	\sum_{n_1+n_2=n,|n_1|<|n|=N(a)}\lV \d^{n_1}\d_su^{(a)}\cdot\nab\d^{n_2}\phi\rV_{L^2}\lesssim & \sum_{n_1+n_2=n,|n_1|<|n|} \lV \d^{n_1}\d_s u^{(a)}\rV_{L^2} \lV\nab\d^{n_2}\phi\rV_{L^{\infty}}\\
	\lesssim &\ep_1\<s\>^{-1+4\delta}\sum_{|l|\leq |a|}\lV \nab u^{(l)}\rV_{H^{N(a)}}+\ep_1^3\<s\>^{-2+7\delta}.
	\end{align*}
	If $|a|\leq N_1-2$, from (\ref{Dec_dtv_VF}) and (\ref{Dec_Phi}) we have
	\begin{align*}
	&\sum_{n_1+n_2=n,|n_1|<|n|=N(a)}\lV \d^{n_1}\d_su^{(a)}\cdot\nab\d^{n_2}\phi\rV_{L^2}\\
	\lesssim &\lV\<\nab\>^{N(a)/2}\d_s u^{(a)}\rV_{L^{\infty}}\lV \nab\phi\rV_{H^{N(a)}}+\lV\d_s u^{(a)}\rV_{H^{N(a)-1}}\lV \<\nab\>^{N(a)/2}\nab\phi\rV_{L^{\infty}}\\
	\lesssim & \ep_1^2 \<s\>^{-5/4+(4+|a|)\delta}+\ep_1\<s\>^{-1+4\delta}\sum_{l\leq a}\lV \nab u^{(l)}\rV_{H^{N(a)}}.
	\end{align*}
	Using these it follows that
	\begin{equation}          \label{II*2}
	\int_0^t \tilde{II}^{a}_{22}(s)ds \lesssim \int_0^t \lV\d_s\phi^{(a)}\rV_{H^{N(a)}}\sum_{n_1+n_2=n,|n_1|<|n|=N(a)}\lV \d^{n_1}\d_su^{(a)}\cdot\nab\d^{n_2}\phi\rV_{L^2}ds\lesssim \ep_1^3.
	\end{equation}

	Finally, we consider the term $\tilde{II}^{a}_{23}$. When $ N_1-1\leq |a|\leq N_1$, $|b|\geq |c|$, using (\ref{H^n_dtv}) and (\ref{Dec_Phi}), it's easy to obtain
	\begin{align*}
	\sum_{b+c=a;|b|<|a|}\lV \d^{n}(\d_s u^{(b)}\cdot\nab \phi^{(c)})\rV_{L^2}\lesssim & \sum_{b+c=a;|b|<|a|}\lV \d_s u^{(b)}\rV_{H^{N(a)}}\sum_{k}2^{N(a)k^+}\lV P_k\nab\phi^{(c)}\rV_{L^{\infty}}\\
	\lesssim& \sum_{|l|\leq |b|}\ep_1\<s\>^{-1+5\delta}\lV \nab u^{(l)}\rV_{H^{N(b)}}+\ep_1^3\<s\>^{-2+8\delta}.
	\end{align*}
	When $N_1-1\leq |a|\leq N_1$, $|b|<|c|$, using (\ref{Dec_dtv_VF}), it follows that
	\begin{equation*}
	\sum_{b+c=a;|b|<|a|}\lV \d^{n}(\d_s u^{(b)}\cdot\nab \phi^{(c)})\rV_{L^2}\lesssim \sum_{b+c=a;|b|<|a|}\sum_{k}2^{N(a)k^+}\lV P_k\d_s u^{(b)}\rV_{L^{\infty}} \lV\nab\phi^{(c)}\rV_{H^{N(a)}}\lesssim \ep_1^2 \<s\>^{-6/5}.
	\end{equation*}
	If $|a|\leq N_1-2$, by (\ref{Dec_dtv_VF}), (\ref{H^n_dtv}) and (\ref{Dec_Phi}), we have
	\begin{align*}
	\sum_{b+c=a;|b|<|a|}\lV \d^{n}(\d_s u^{(b)}\cdot\nab \phi^{(c)})\rV_{L^2}\lesssim &\sum_{b+c=a;|b|<|a|}(\lV\d_s u^{(b)}\rV_{H^{N(a)}}\lV\nab\phi^{(c)}\rV_{L^{\infty}}+\lV\d_s u^{(b)}\rV_{L^{\infty}}\lV\nab\phi^{(c)}\rV_{H^{N(a)}})\\
	\lesssim & \ep_1^2\<s\>^{-6/5}+\sum_{|l|\leq |b|}\ep_1\<s\>^{-1+5\delta}\lV \nab u^{(l)}\rV_{H^{N(b)}}.
	\end{align*}
	By the above bounds, we get
	\begin{equation}        \label{II*3}
	\int_0^t \tilde{II}^{a}_{23}(s)ds\lesssim \sum_{b+c=a,|b|<|a|}\sum_{|n|=N(a)}\int_0^t \lV\d_s\phi^{(a)}\rV_{H^{N(a)}}\lV \d^{n}(\d_s u^{(b)}\cdot\nab \phi^{(c)})\rV_{L^2}ds\lesssim \ep_1^3.
	\end{equation}
	
	Hence, from (\ref{II*1})-(\ref{II*3}), we have
	\begin{equation*}
	\int_0^t \widetilde{II}^{a}_2(s)ds\lesssim \ep_1^3\<t\>^{\ka(a)\delta},\ \mathrm{for}\ b+c=a.
	\end{equation*}
	This completes the proof of the bound (\ref{Ephi_II2}).
	
	\emph{Step 3. We prove the following two bounds }
	\begin{equation}         \label{Ephi_HignOrd}
	\int_0^t |II^{a}_3(s)|+|II^a_4(s)|d\lesssim \ep_1^4\<t\>^{2\ka(a)\delta}.
	\end{equation}

	By $\div u=0$, we rewrite $II^{a}_3$ as
	\begin{equation}      \label{Ephi_II3_Rewr}
	\begin{aligned}
	II^{a}_3=&\sum_{|n|=N(a)}\int_{\R^3}\d^{n}\d_s\phi^{(a)}\cdot(u\cdot\nab(u\cdot\nab\d^{n}\phi^{(a)}))dx\\
	&+\sum_{n_1+n_2+n_3=n,|n_3|<|n|=N(a)}\int_{\R^3}\d^{n}\d_s\phi^{(a)}\cdot\d^{n_1}u\cdot\nab(\d^{n_2}u\cdot\nab\d^{n_3}\phi^{(a)})dx\\
	&+\sum_{b+c+e=a,|e|<|a|}\sum_{|n|=N(a)}C_a^{b,c}\int_{\R^3}\d^n \d_s\phi^{(a)}\cdot\d^n(u^{(b)}\cdot\nab(u^{(c)}\cdot\nab\phi^{(e)}))dx\\
	&+\sum_{b+c+e=a}C_a^{b,c}\int_{\R^3}\d_s\phi^{(a)}\cdot (u^{(b)}\cdot\nab(u^{(c)}\cdot\nab\phi^{(e)})) dx.
	\end{aligned}
	\end{equation}
	The first term in the right hand side of (\ref{Ephi_II3_Rewr}) can be estimated using $\div u=0$ and (\ref{H^n_dtv}),
	\begin{align*}
	&\int_0^t\int_{\R^3}\d^{n}\d_s\phi^{(a)}\cdot(u\cdot\nab(u\cdot\nab\d^n\phi^{(a)}))dxds\\
	=& \int_0^t \int_{\R^3}-\frac{1}{2}\d_s (u\cdot\nab\d^{N(a)}\phi^{(a)})+(\d_s u\cdot\nab\d^{N(a)}\phi^{(a)})\cdot (u\cdot\nab\d^{N(a)}\phi^{(a)})dxds\\
	\lesssim & \sup_{s\in[0,t]}\lV u(s)\rV_{H^2}^2\lV\nab\phi^{(a)}(s)\rV_{H^{N(a)}}^2+\int_0^t \lV\d_s u\rV_{H^2}\lV \nab u\rV_{H^1}\lV\nab\phi^{(a)}\rV_{H^{N(a)}}^2 ds\\
	\lesssim & \ep_1^4\<t\>^{2\ka(a)\delta}.
	\end{align*}
	The other three terms in (\ref{Ephi_II3_Rewr}) can be estimated similarly by H\"{o}lder and (\ref{Main_Prop_Ass1}). The bound (\ref{Ephi_HignOrd}) for $II^a_4$ can be estimated directly, using (\ref{Dec_Phi}) and (\ref{Dec_Phi/nab}). Hence, the desired bound (\ref{Ephi_HignOrd}) follows. This completes the proof of the Proposition. 	
\end{proof}

\section{Bounds on the profile: weighted $L^2$ norms}
In this section we prove (\ref{Main_Prop_result2}), namely,
\begin{proposition} \label{Prop_Psia}
	With the hypothesis in Proposition \ref{Main_Prop}, for any $t\in[0,T]$, $0\leq |a|\leq N_1-1$, we have
	\begin{equation}           \label{ImpBd_dxiPsi}
	\lV\FF^{-1}(|\xi|\nab_{\xi}\widehat{\Psi^{(a)}})\rV_{H^{N(|a|+1)}}\lesssim \ep_0\<t\>^{\ka(|a|+1)\delta}.
	\end{equation}
\end{proposition}
\begin{proof}
	By (\ref{xidxif}), $[e^{it|\nab|},S]=[e^{it|\nab|},\Om]=0$ and (\ref{Main_Prop_result1}), we have
	\begin{equation}                \label{keyIneq}
	\lV\mathcal{F}^{-1}(|\xi|\nab_{\xi}\widehat{\Psi^{(a)}}(\xi))\rV_{H^N}\lesssim \ep_0\<t\>^{\ka(|a|+1)\delta}+\lV t\d_t \Psi^{(a)}\rV_{H^N},
	\end{equation}
	Then it suffices to estimate the last term in the right-hand side of (\ref{keyIneq}). Using $\phi$-equation in (\ref{Main_Sys_VecFie}) and (\ref{Psi-Phi-relat}), it suffices to prove
	\begin{gather}        \label{dtPsi_Term1}
	\lV \d_t u^{(b)}\cdot\nab\phi^{(c)}\rV_{H^{N(|a|+1)}}\lesssim \ep_1^2\<t\>^{-1},\\	\label{dtPsi_Term2}
	\lV u^{(b)}\cdot\nab\d_t\phi^{(c)}\rV_{H^{N(|a|+1)}}\lesssim \ep_1^2\<t\>^{-1+\ka(|a|+1)\delta},
	\end{gather}
	for $b+c=a$, $k\in\Z$, and
	\begin{gather} \label{dtPsi_Term3}
	\lV u^{(b)}\cdot\nab(u^{(c)}\cdot\nab\phi^{(e)})\rV_{H^{N(|a|+1)}}\lesssim \ep_1^3\<t\>^{-1},    \\     \label{dtPsi_Term4}
	\lV \frac{\Phi^{(b)}}{|\nab|}\Phi^{(c)}\Phi^{(e)}\rV_{H^{N(|a|+1)}}\lesssim \ep_1^3\<t\>^{-1},
	\end{gather}
	for $b+c+e=a$, $k\in\Z$.
		
	\emph{Step 1: Proof of (\ref{dtPsi_Term1}).}
	
	When $|a|= N_1-1$, $|b|\geq |c|$, by (\ref{Hntdtv(a)}) and (\ref{Dec_Phi}), we have
	\begin{align*}
	\lV \d_t u^{(b)}\cdot\nab\phi^{(c)}\rV_{H^{N(|a|+1)}}\lesssim  \lV \d_t u^{(b)}\rV_{H^{N(|a|+1)}}\sum_k 2^{N(|a|+1)k^+}\lV P_k \nab\phi^{(c)}\rV_{L^{\infty}}
	\lesssim  \ep_1^2 \<t\>^{-3/2+(6+|b|)\delta}.
	\end{align*}
	When $|a|= N_1-1$, $|b|<|c|$, using (\ref{Dec_dtv_VF}), it follows that
	\begin{equation*}
	\lV \d_t u^{(b)}\cdot\nab\phi^{(c)}\rV_{H^{N(|a|+1)}}\lesssim  (\sum_{k\geq 0}2^{N(|a|+1)k}\lV \d_t u^{(b)}\rV_{L^{\infty}}+\lV \d_t u^{(b)}\rV_{L^{\infty}}) \lV\nab \phi^{(c)}\rV_{H^{N(|a|+1)}}\lesssim \ep_1^2 \<t\>^{-5/4+(6+|b|)\delta},
	\end{equation*}
	When $|a|\leq N_1-2$, we obtain from (\ref{Hntdtv(a)}), (\ref{Dec_Phi}) and (\ref{Dec_dtv_VF})
	\begin{align*}
	\lV \d_t u^{(b)}\cdot\nab\phi^{(c)}\rV_{H^{N(|a|+1)}}\lesssim \lV\d_t u^{(b)}\rV_{H^{N(|a|+1)}}\lV \nab \phi^{(c)}\rV_{L^{\infty}}+\lV\d_t u^{(b)}\rV_{L^{\infty}}\lV \nab \phi^{(c)}\rV_{H^{N(|a|+1)}}\lesssim \ep_1^2 \<t\>^{-5/4+(6+|b|)\delta},
	\end{align*}
	Hence, the bound (\ref{dtPsi_Term1}) follows.

	\emph{Step 2: Proof of (\ref{dtPsi_Term2}).} If $|a|=|b|=N_1-1,c=0$. (\ref{Dec_Phi}) implies
	\begin{equation}             \label{dtPsi_Term2_case1}
	\lV u^{(b)}\cdot \nab\d_t\phi\rV_{H^{N(|a|+1)}}\lesssim  \lV u^{(|b|)}\rV_{H^{N(|a|+1)}} \sum_{k}2^{N(|a|+1)k^++k}\lV \d_t\phi\rV_{L^{\infty}}\lesssim \ep_1^2\<t\>^{-1+\ka(|a|+1)\delta}.
	\end{equation}
    Now it suffices to consider the case $|b|\leq N_1-2, |c|\geq 0$.

    Decomposing dyadically in frequency, we have
	\begin{align*}
	\lV u^{(b)}\cdot\nab\d_t\phi^{(c)}\rV_{H^{N(|a|+1)}}\lesssim (\sum_k 2^{2N(|a|+1)k^+}I_k^2)^{1/2},
	\end{align*}
	where
	\begin{equation*}
	I_k:=\lV\sum_{k_1,k_2}\varphi_k(\xi)\int_{\R^3} e^{it|\eta|}\eta \widehat{P_{k_1}u^{(b)}}(\xi-\eta)\widehat{P_{k_2}\Psi^{(c)}}(\eta)d\eta\rV_{L^2}.
	\end{equation*}
	We further divided $I_k$ into high-low, low-high, high-high case,
	\begin{equation*}
	I_k\lesssim I^{hl}_k+I^{lh}_k+I^{hh}_k,
	\end{equation*}
	where
	\begin{gather*}
	I^{hl}_k:=\sum_{k_2\leq k-5}\lV\varphi_k(\xi)\int_{\R^3} e^{it|\eta|}\eta \widehat{P_{[k-3,k+3]}u^{(b)}}(\xi-\eta)\widehat{P_{k_2}\Psi^{(c)}}(\eta)d\eta\rV_{L^2},\\
	I^{lh}_k:=\lV\varphi_k(\xi)\int_{\R^3} e^{it|\eta|}\eta \widehat{P_{<k+6}u^{(b)}}(\xi-\eta)\widehat{P_{[k-5,k+5]}\Psi^{(c)}}(\eta)d\eta\rV_{L^2},\\
	I^{hh}_k:=\sum_{k_1,k_2;k_2\geq k+5}\lV\varphi_k(\xi)\int_{\R^3} e^{it|\eta|}\xi \widehat{P_{[k_2-3,k_2+3]}u^{(b)}}(\xi-\eta)\widehat{P_{k_2}\Psi^{(c)}}(\eta)d\eta\rV_{L^2}.
	\end{gather*}
	
	\emph{Step 2.1. The contribution of $I^{hl}_k$.} Integration by parts in $\eta$ yields
	\begin{align*}
	I^{hl}_k\lesssim &  t^{-1}\sum_{k_2<k-5}\Big[\lV\int_{\R^3} \d_{\eta_j}\big(\frac{\eta_j}{|\eta|}\eta\big)\widehat{P_{[k-3,k+3]} u^{(b)}}(\xi-\eta) \widehat{P_{k_2}\Phi^{(c)}}(\eta) d\eta\rV_{L^2}\\
	&+\lV \int_{\R^3} \frac{\eta_j}{|\eta|}\eta \d_{\eta_j} \widehat{P_{[k-3,k+3]} u^{(b)}}(\xi-\eta) \widehat{P_{k_2}\Phi^{(c)}}(\eta) d\eta\rV_{L^2}\\
	&+\lV \int_{\R^3} \frac{\eta_j}{|\eta|}\eta  \widehat{P_{[k-3,k+3]} u^{(b)}}(\xi-\eta) e^{it|\eta|}\d_{\eta_j}\widehat{P_{k_2}\Psi^{(c)}}(\eta) d\eta\rV_{L^2}\Big].
	\end{align*}
	When $|a|=|b|\leq N_1-2, c=0$, by Lemma $\ref{Sym_Pre_lemma}$, (\ref{Main_Prop_Ass1}), (\ref{Main_Prop_Ass2}) and (\ref{nab_xi_v(b,c)_HN}) we have
	\begin{align*}
	\big[\sum_k 2^{2N(|a|+1)k^+}(I^{hl}_k)^2\big]^{1/2}
	\lesssim & t^{-1}  [\lV u^{(b)}\rV_{H^{N(|a|+1)}}\lV \Phi^{(c)}\rV_{H^2}+\lV |\xi|\d_{\xi}\widehat{u^{(b)}}(\xi)\rV_{H^{N(|a|+1)}}\sum_{k_2}\lV \Phi^{(c)}\rV_{L^{\infty}}\\
	&+\lV v^{(b)}\rV_{H^{N(|a|+1)}}\lV|\xi|\d_{\xi}\widehat{\Psi^{(c)}}\rV_{H^{N(|a|+1)}}]\\
	\lesssim & t^{-1}(\ep_1^2\<t\>^{\ka(c)\delta}+\ep_1^2\<t\>^{-1/3}+\ep_1^2\<t\>^{\ka(|c|+1)\delta})\\
	\lesssim & \ep_1^2\<t\>^{-1+\ka(|c|+1)\delta},
	\end{align*}
	When $|b|\leq N_1-2,|c|\geq 1$,	it follows from (\ref{Main_Prop_Ass1}), (\ref{Main_Prop_Ass2}) and (\ref{rv}) that
	\begin{align*}       \label{dtPsi_Term2_case2}
	\big[\sum_k 2^{2N(|a|+1)k^+}(I^{hl}_k)^2\big]^{1/2}
	\lesssim &t^{-1}[\lV u^{(b)}\rV_{H^{N(|a|+1)}}\lV \Phi^{(c)}\rV_{H^2}+\{2^{N(|a|+1)k^+}\lV rP_{[k-3,k+3]}u^{(b)}\rV_{L^{\infty}}\}_{l^2}\lV \Phi^{(c)}\rV_{H^2}\\
	&+\lV u^{(b)}\rV_{H^{N(|a|+1)}}\lV|\xi|\d_{\xi}\Psi^{(c)}\rV_{H^{N(|a|+1)}}]\\
	\lesssim & t^{-1}(\ep_1^2\<t\>^{\ka(c)\delta}+\sum_{\alpha_1\leq 1,\alpha_2\leq 2}\lV \d_r\tilde{\Om}^{\alpha_1}v^{(b)}\rV_{H^{N(|a|+1)}}^{1/2}\lV \tilde{\Om}^{\alpha_2}v^{(b)}\rV_{H^{N(|a|+1)}}^{1/2}\ep_1\<t\>^{\ka(c)\delta}\\
	&+\ep_1^2\<t\>^{\ka(|c|+1)\delta})\\
	\lesssim &\ep_1^2 \<t\>^{-1+\ka(|a|+1)\delta},
	\end{align*}
	Therefore, we have
	\begin{equation}\label{hl}
	\big[\sum_k 2^{2N(|a|+1)k^+}(I^{hl}_k)^2\big]^{1/2}\lesssim \ep_1^2\<t\>^{-1+\ka(|a|+1)\delta}.
	\end{equation}
	
	\emph{Step 2.2: The contribution of $I^{lh}_k$ and $I^{hh}_k$.} We only give the bound $I^{lh}_k$, $I^{hh}_k$ is estimated similarly. Integration by parts in $\eta$, it follows that
	\begin{align*}
	I^{lh}_k\lesssim &  t^{-1}\Big[\lV\int_{\R^3} \d_{\eta_j}\big(\frac{\eta_j}{|\eta|}\eta\big)\widehat{P_{<k+6} u^{(b)}}(\xi-\eta) \widehat{P_{[k-5,k+5]}\Phi^{(c)}}(\eta) d\eta\rV_{L^2}\\
	&+\lV \int_{\R^3} \frac{\eta_j}{|\eta|}\eta \d_{\eta_j} \widehat{P_{<k+6} u^{(b)}}(\xi-\eta) \widehat{P_{[k-5,k+5]}\Phi^{(c)}}(\eta) d\eta\rV_{L^2}\\
	&+\lV \int_{\R^3} \frac{\eta_j}{|\eta|}\eta  \widehat{P_{<k+6} u^{(b)}}(\xi-\eta) e^{it|\eta|}\d_{\eta_j}\widehat{P_{[k-5,k+5]}\Psi^{(c)}}(\eta) d\eta\rV_{L^2}\Big]
	\end{align*}
	Then by Lemma $\ref{Sym_Pre_lemma}$, (\ref{Main_Prop_Ass1}), (\ref{Main_Prop_Ass2}) and (\ref{rv}), we have
	\begin{equation}\label{lhhh}
	\begin{aligned}
	\big[\sum_k 2^{2N(|a|+1)k^+}(I^{lh}_k)^2\big]^{1/2}
	\lesssim & t^{-1}  [\lV u^{(b)}\rV_{H^{N(|a|+1)}}\lV \Phi^{(c)}\rV_{H^2}+\lV ru^{(b)}\rV_{L^{\infty}}\lV \Phi^{(c)}\rV_{H^{N(|a|+1)+1}}\\
	&+\lV u^{(b)}\rV_{H^{N(|a|+1)}}\lV|\xi|\d_{\xi}\widehat{\Psi^{(c)}}\rV_{H^{N(|a|+1)}}]\\
	\lesssim & t^{-1}(\ep_1^2 \<t\>^{\ka(c)\delta}+\ep_1^2\<t\>^{\ka(|c|+1)\delta})\\
	\lesssim & \ep_1^2 \<t\>^{-1+\ka(|a|+1)\delta}.
	\end{aligned}
	\end{equation}
	
	From (\ref{hl}) and (\ref{lhhh}), the bound (\ref{dtPsi_Term2}) for $|b|\leq N_1-2, |c|\geq 0$ is obtained immediately. This completes the proof of (\ref{dtPsi_Term2}).
	
	\emph{Step 3: Proof of (\ref{dtPsi_Term3}).}
	
	When $|e|\leq N_1-2$, using (\ref{Dec_v_VecFie}) and (\ref{Dec_Phi}), we get
	\begin{align*}
	\lV u^{(b)}\cdot\nab(u^{(c)}\cdot\nab\phi^{(e)})\rV_{H^{N(|a|+1)}}
	\lesssim &\lV u^{(b)}\rV_{H^{N(|a|+1)}}(\lV\nab u^{(c)}\rV_{L^{\infty}}\lV \nab\phi^{(e)}\rV_{L^{\infty}}+\lV u^{(c)}\rV_{L^{\infty}}\lV \nab^2\phi^{(e)}\rV_{L^{\infty}})\\
	&+\lV u^{(b)}\rV_{L^{\infty}}(\lV u^{(c)}\rV_{H^{N(|a|+1)+1}}\lV \nab\phi^{(e)}\rV_{L^{\infty}}+\lV u^{(c)}\rV_{L^{\infty}}\lV \nab\phi^{(e)}\rV_{H^{N(|a|+1)+1}})\\
	\lesssim & \ep_1^3\<t\>^{-5/4},
	\end{align*}
	When $|e|\geq N_1-1$, by (\ref{Dec_v_VecFie}), we have
	\begin{align*}
	\lV u\cdot\nab(u\cdot\nab\phi^{(e)})\rV_{H^{N(|a|+1)}}
	\lesssim & (\sum_{k\geq 0}2^{N(|e|+1)k+k}\lV P_k u\rV_{L^{\infty}}+\lV P_{<0}u\rV_{L^{\infty}})^2 \lV \nab\phi^{(e)}\rV_{H^{N(|e|+1)+1}}\lesssim \ep_1^3\<t\>^{-5/4},
	\end{align*}
	 Hence, the bound (\ref{dtPsi_Term3}) follows.
	
	Finally, the bound (\ref{dtPsi_Term4}) is an consequence of (\ref{Dec_Phi}) and (\ref{Dec_Phi/nab}).
	This completes the proof of the proposition.
\end{proof}

\appendix

\section{Deriving the system (\ref{Main_Sys}), (\ref{Main_Sys-v}) and (\ref{Main_Sys_VecFie})}
\subsection{Deriving the systems (\ref{Main_Sys}) and (\ref{Main_Sys-v}) from (\ref{ori_sys})}       \label{A.1}
\

\emph{Step 1: we derive the following system from (\ref{ori_sys})
\begin{equation}                   \label{Sys}
\left\{\begin{aligned}
\partial_t u+\tilde{L} u   & =-u\cdot\nabla u-\nab p-\d_j(\nab \phi\cdot \d_j \phi)+\nab\otimes(\phi\otimes\nab u)+Err11+Err12), \\
\div u & =0,\\
\d_t^2 \phi-\Delta\phi&=-\d_tu\cdot\nab \phi-2u\cdot\nab\d_t\phi-u\cdot\nab(u\cdot\nab\phi)+Err2
\end{aligned}
\right.
\end{equation}
where the two order linear operator $\tilde{H}$ is defined by
\begin{equation*}
\tilde{L}u:=-\frac{\nu_4}{2}\Delta u-(\frac{\nu_5}{2}\Delta u_1+(\nu_1+\nu_5)\d_1^2 u_1,\frac{\nu_5}{2}(\d_1^2 u_2+\d_1\d_2 u_1),\frac{\nu_5}{2}(\d_1^2 u_3+\d_1\d_3 u_1))^{\top}.
\end{equation*}}

Firstly, we proceed to derive the equations satisfied by $\phi$. From the third component of $d$-equation in (\ref{ori_sys}) and (\ref{d_pres}), by directly computation we have
\begin{equation}  \label{phi2}
\ddot{\phi}_2-\Delta \phi_2= \frac{1}{2}\sin 2\phi_2(-\dot{\phi}_1^2+|\nab\phi_1|^2).
\end{equation}
After taking the first component of $d$-equation of the system (\ref{ori_sys}), by (\ref{phi2}) we obtain
\begin{equation}            \label{phi1}
\ddot{\phi}_1-\Delta \phi_1=2\tan\phi_2(\dot{\phi}_1\dot{\phi}_2-\nab\phi_1\nab\phi_2).
\end{equation}
Hence, the $\phi$-equation in (\ref{Sys}) follows.

Now we derive the $u$-equation in (\ref{Sys}). By (\ref{d_pres}), the $i$-th, $1\leq i\leq 3$ component of $\div (\nab d\odot\nab d)$ in the first equation of the system (\ref{ori_sys}) can be rewritten as
\begin{equation*}
\sum_j\d_j (\d_i d\cdot\d_j d)=\sum_j \d_j(\d_i\phi\d_j\phi)-\sum_j \d_j(\sin^2\phi_2\d_i\phi_1\d_j\phi_1).
\end{equation*}

Since the orientation field $d$ is near $\vec{i}=(1,0,0)$. By (\ref{d_pres}) and Taylor series expansion we have
\begin{align*}
(d_kd_pd_i d_j)(\phi_1,\phi_2)=&(d_kd_pd_i d_j)(0,0)+(\phi_1\d_1+\phi_2\d_2)(d_kd_pd_i d_j)(0,0)\\
&+\int_0^1 (\phi_1\d_1+\phi_2\d_2)^2(d_kd_pd_i d_j)(s\phi_1,s\phi_2)(1-s)ds,
\end{align*}
and
\begin{align*}
(d_id_k)(\phi_1,\phi_2)=&(d_id_k)(0,0)+(\phi_1\d_1+\phi_2\d_2)(d_id_k)(0,0)\\
&+\int_0^1 (\phi_1\d_1+\phi_2\d_2)^2(d_id_k)(s\phi_1,s\phi_2)(1-s)ds.
\end{align*}
Then for $\div \sigma$, from the above two expression we obtain the linear term 
\begin{equation*}
\nu_1\left(\begin{array}{ccc}
\d_1 A_{11}\\
0\\
0
\end{array}\right)+\nu_5\left(\begin{array}{ccc}
\d_1A_{11}+\d_jA_{ij}\\
\d_1A_{12}\\
\d_1A_{13}
\end{array}\right)=
\left(
\begin{array}{ccc}
\frac{\nu_5}{2}\Delta u_1+(\nu_1+\nu_5)\d_1^2 u_1\\
\frac{\nu_5}{2}(\d_1^2 u_2+\d_1\d_2 u_1)\\
\frac{\nu_5}{2}(\d_1^2 u_3+\d_1\d_3 u_1)
\end{array}
\right),
\end{equation*}
the quadratic terms
\begin{align*}
\nab\otimes(\phi\otimes\nab u):=&\nu_1\left(\begin{array}{ccc}
\d_2(A_{11}\phi_1)+\d_3(A_{11}\phi_2)+2\d_1(\phi_1A_{12}+\phi_2A_{13})\\
\d_1(A_{11}\phi_1)\\
\d_1(A_{11}\phi_2)
\end{array}\right)\\
&+\nu_5\left(\begin{array}{ccc}
\d_2(\phi_1A_{11})+\d_3(\phi_2A_{11})+\d_1(\phi_1A_{21}+\phi_2A_{31})+\d_j(\phi_1A_{2j}+\phi_2A_{3j})\\
\d_2(\phi_1A_{12})+\d_3(\phi_2A_{12})+\d_1(\phi_1A_{22}+\phi_2A_{32})+\d_j(\phi_1A_{1j})\\
\d_2(\phi_1A_{13})+\d_3(\phi_2A_{13})+\d_1(\phi_1A_{23}+\phi_2A_{33})+\d_j(\phi_2A_{1j})
\end{array}\right).
\end{align*}
and the error terms 
\begin{equation*}
Err11=(Err11_1,Err11_2,Err11_3)^{\top},
\end{equation*}
where
\begin{align*}
Err11_i=&\nu_1\d_j (A_{kp}\int_0^1 (\phi_1\d_1+\phi_2\d_2)^2(d_kd_pd_i d_j)(s\phi_1,s\phi_2)(1-s)ds)\\
&+\nu_5\d_j(A_{ki}\int_0^1 (\phi_1\d_1+\phi_2\d_2)^2(d_jd_k)(s\phi_1,s\phi_2)(1-s)ds)\\
&+\nu_5\d_j(A_{kj}\int_0^1 (\phi_1\d_1+\phi_2\d_2)^2(d_id_k)(s\phi_1,s\phi_2)(1-s)ds).
\end{align*}
Thus the $u$-equation in (\ref{Sys}) follows, and hence we obtain the system (\ref{Sys}).
	
\emph{Step 2: We derive the $v$-equation in (\ref{Main_Sys-v}) from $u$-equation in (\ref{Sys}).}

Applying the Leray projection $\P$ to $u$-equation in (\ref{Sys}), we obtain
\begin{equation}   \label{u_eq}
\d_t u+\P\tilde{L} u   =\P(-u\cdot\nabla u-\d_j(\nab \phi\cdot \d_j \phi)+\mu\d_j(\phi\otimes\nab u)+Err11+Err12),
\end{equation}
where $\P\tilde{L}$ is a two order operator, i.e
\begin{align*}
\bar{L} u:=\P\tilde{L} u=&-\frac{\nu_4}{2}\Delta u-\nu_1\frac{\d_1^2}{|\nab|^2}\big(-(\d_2^2+\d_3^2)u_1,\d_1\d_2 u_1,\d_1\d_3u_1\big)^{\top}\\
&-\nu_5\big(\frac{1}{2}\Delta u_1,\frac{1}{2}(\d_1^2 u_2-\d_1\d_2 u_1),\frac{1}{2}(\d_1^2 u_3-\d_1\d_3 u_1)\big)^{\top}.
\end{align*}
By $\div u=0$, we further have
\begin{equation*}
\bar{L}u=
\left(\begin{array}{ccc}
-\frac{\nu_4+\nu_5}{2}\Delta+\nu_1\frac{\d_1^2(\d_2^2+\d_3^2)}{|\nab|^2} & 0&0\\
0&-\frac{\nu_4+\nu_5}{2}\Delta+\frac{\nu_5}{2}\d_3^2+\nu_1\frac{\d_1^2\d_2^2}{|\nab|^2}& -(\frac{\nu_5}{2}-\nu_1\frac{\d_1^2}{|\nab|^2})\d_2\d_3\\
0& -(\frac{\nu_5}{2}-\nu_1\frac{\d_1^2}{|\nab|^2})\d_2\d_3&-\frac{\nu_4+\nu_5}{2}\Delta+\frac{\nu_5}{2}\d_2^2+\nu_1\frac{\d_1^2\d_3^2}{|\nab|^2}
\end{array}\right)
\left(\begin{aligned}
u_1\\
u_2\\
u_3
\end{aligned}\right).
\end{equation*}
In order to diagonalize $u$-equation (\ref{u_eq}), let
\begin{equation}
v=\U u,
\end{equation}
where the operator $\U$ is
\begin{equation}          \label{U-apend}
\U:=\left(\begin{array}{ccc}
1&0&0\\
0&\frac{-i\d_2}{\sqrt{-\d_2^2-\d_3^2}}&\frac{-i\d_3}{\sqrt{-\d_2^2-\d_3^2}}\\
0&\frac{-i\d_3}{\sqrt{-\d_2^2-\d_3^2}}&\frac{i\d_2}{\sqrt{-\d_2^2-\d_3^2}}
\end{array}\right).
\end{equation}
From the definition of $\U$, we can recover $u$ by
$$u=\U v.$$ 
Then applying $\U$ to (\ref{u_eq}), we obtain
\begin{equation*}
\d_t v+\U\bar{L} \U v    =\U\P(-u\cdot\nabla u-\d_j(\nab \phi\cdot \d_j \phi)+\nab\otimes(\phi\otimes\nab u)+Err11+Err12).
\end{equation*}
Using Fourier transformation, the operator $\U\bar{L} \U$ can be rewritten as
\begin{align*}
&\mathcal{F}(\U\bar{L} \U v)(\xi)=U(\xi)\bar{L}(\xi)U(\xi)\widehat{v}(\xi)=\widehat{L v}(\xi)\\
=& \Big((\frac{\nu_4+\nu_5}{2}|\xi|^2+\nu_1\frac{\xi_1^2(\xi_2^2+\xi_3^2)}{|\xi|^2})\widehat{v}_1(\xi),(\frac{\nu_4+\nu_5}{2}|\xi|^2+\nu_1\frac{\xi_1^2(\xi_2^2+\xi_3^2)}{|\xi|^2})\widehat{v}_2(\xi),(\frac{\nu_4}{2}|\xi|^2+\frac{\nu_5}{2}\xi_1^2)\widehat{v}_3(\xi)\Big).
\end{align*}
Thus the $v$-equation in (\ref{Main_Sys-v}) follows.

\subsection{Deriving the system (\ref{Main_Sys_VecFie})}
\
By the standard argument, from (\ref{Sys}) we can derive that
\begin{equation}         \label{Sys_VF}
\left\{\begin{aligned}
\d_t u^{(a)}+\tilde{L} u^{(a)}+L_1(u)=&-\sum_{b+c=a}C_a^b u^{(b)}\cdot\nab u^{(c)}-\nab p^{(a)}\\
&-\sum_{b+c=a}C_a^b \d_j(\nab(S-1)^{b_1}\Gamma^{b'}\phi\cdot\d_j(S-1)^{c_1}\Gamma^{c'}\phi)\\
&+(S+1)^{a_1}\Gamma^{a'}\nab\otimes(\phi\otimes\nab u)+(S+1)^{a_1}\Gamma^{a'}( Err11+Err12)),\\
\div u^{(a)}=&0,\\
\d_t^2\phi^{(a)}-\Delta\phi^{(a)}=&-\sum_{b+c=a}C_a^b(\d_t u^{(b)}\cdot\nab\phi^{(c)}+2u^{(b)}\cdot\nab\d_t\phi^{(c)})\\
&-\sum_{b+c+e=a}C_a^{b,c}u^{(b)}\cdot\nab(u^{(c)}\cdot\nab\phi^{(e)})+(S+2)^{(a_1)}\Gamma^{(a')}Err2,
\end{aligned}
\right.
\end{equation}
then by $\div u^{(a)}=0$, $v^{(a)}=\U u^{(a)}$ and the same argument as Step 2 in Appendix \ref{A.1}, applying $\U\P$ to $u^{(a)}$-equation in (\ref{Sys_VF}), the system (\ref{Main_Sys_VecFie}) is obtained.

\end{document}